\documentclass[fleqn]{amsproc}

\newcommand{\maincounter}{subsection}

\usepackage[T1]{fontenc}
\usepackage{newtxtext}

\usepackage[%
cal=cm,
bb=ams,
frak=euler,
]{mathalpha}
\usepackage{mathrsfs}
\usepackage{bm}
\providecommand{\mathbold}{\bm}
\usepackage[cmintegrals,upint,varvw]{newtxmath} 
\providecommand{\batchrelax}[1]{
\renewcommand{\do}[1]{\expandafter\let\csname##1\endcsname\relax}\docsvlist{#1}}

\DeclareSymbolFont{largesymbolsCM}{OMX}{cmex}{m}{n}
\batchrelax{txprod,prod,txcprod,coprod,txsum,sum}

\DeclareMathSymbol{\prod}{\mathop}{largesymbolsCM}{"51} 

\DeclareMathSymbol{\coprod}{\mathop}{largesymbolsCM}{"60} 

\let\sum\relax
\DeclareMathSymbol{\sum}{\mathop}{largesymbolsCM}{"50} 
\DeclareMathSymbol{\bigcupop}{\mathop}{largesymbolsCM}{"53} 
\DeclareMathSymbol{\bigcapop}{\mathop}{largesymbolsCM}{"54} 
\DeclareMathDelimiter{(}{\mathopen} {operators}{"28}{largesymbolsCM}{"00} 
\DeclareMathDelimiter{)}{\mathclose}{operators}{"29}{largesymbolsCM}{"01} 

\usepackage{geometry}
\newcommand{\colorlinks}{%
\hypersetup{colorlinks=true,%
linkcolor=magenta,%
citecolor=purple,%
}%
}

\usepackage{microtype}
\usepackage{graphicx}

\usepackage[usenames,dvipsnames,table]{xcolor}
\usepackage[most,theorems]{tcolorbox}
\newtcbox{\hlT}[1][blue]{on line,
colback=black!20,colframe=black,
boxsep=1pt,left=2pt,right=2pt,top=1pt,bottom=1pt,enhanced jigsaw,
boxrule=0pt}

\newcommand{\bfit}[2][black]{\textit{\textbf{\textcolor{#1}{#2}}}}

\renewcommand{\emph}[1]{\bfit[blue]{#1}}

\usepackage{etoolbox,mathtools}

\usepackage[bookmarksdepth=3,bookmarksnumbered,%
bookmarksopenlevel=0,%
bookmarksopen=true,color links=true,linkcolor=magenta,%
backref=true,
]{hyperref}
\usepackage[capitalise,nameinlink,nosort]{cleveref}

\crefdefaultlabelformat{#2#1#3} 
\creflabelformat{equation}{\textbf{#2(#1)}#3} 
\crefname{enumi}{}{}

\usepackage{longtable}
\usepackage{booktabs}
\usepackage{array,tabularx,colortbl}
\usepackage{multicol,multirow}

\providecommand{\batchrelax}[1]{
\renewcommand{\do}[1]{\expandafter\let\csname##1\endcsname\relax}\docsvlist{#1}}
\providecommand{\Batchrelax}[2][]{
\renewcommand{\do}[1]{\expandafter\let\csname#1##1\endcsname\relax}\docsvlist{#2}}

\usepackage{mathdots}

\newcommand{\define}[4]{\expandafter#1\csname#3#4\endcsname{#2{#4}}}
\newcommand{\definefromcommand}[4]{\expandafter#1\csname#3#4\endcsname{#2{\csname#4\endcsname}}} 
\newcommand{\defineoptional}[4]{\expandafter#1\csname#3#4\endcsname[1][]{#2{#4}}} 
\newcommand{\DefineCommand}[3]{%
\forcsvlist{\define{\newcommand}{#1}{#2}}{#3}
}

\newcommand{\DefineMathOperator}[3]{%
\forcsvlist{\define{\DeclareMathOperator}{#1}{#2}}{#3}
}
\newcommand{\DefineCommandFromCommand}[3]{%
\forcsvlist{\definefromcommand{\newcommand}{#1}{#2}}{#3}
}
\newcommand{\DefineOperatorFromCommand}[3]{%
\forcsvlist{\definefromcommand{\DeclareMathOperator}{#1}{#2}}{#3}
}
\newcommand{\DeclareMathOperatorStar}{\DeclareMathOperator*}
\newcommand{\DefineOperatorStarFromCommand}[3]{%
\forcsvlist{\definefromcommand{\DeclareMathOperatorStar}{#1}{#2}}{#3}
}
\newcommand{\DefineCommandOptional}[3]{%
\forcsvlist{\defineoptional{\newcommand}{#1}{#2}}{#3}
}

\newcommand{\colormath}[3]{\textcolor{#1}{\csname#2\endcsname{#3}}}
\newcommand{\ColorMath}[4][]{\forcsvlist{\DefineCommand{\colormath{#2}{#3}}{#1}}{#4}}

\newcommand{\mathfont}[2]{\csname#1\endcsname{#2}}
\newcommand{\MathFonts}[3][]{\forcsvlist{\DefineCommand{\mathfont{#2}}{#1}}{#3}}
\newcommand{\MathOperators}[3][]{\forcsvlist{\DefineMathOperator{\mathfont{#2}}{#1}}{#3}}
\newcommand{\MathSymbols}[3][]{\forcsvlist{\DefineCommandFromCommand{\mathfont{#2}}{#1}}{#3}}
\newcommand{\ModifyMathOperator}[3][]{\forcsvlist{\DefineOperatorFromCommand{\mathfont{#2}}{#1}}{#3}}
\newcommand{\ModifyMathOperatorStar}[3][]{\forcsvlist{\DefineOperatorStarFromCommand{\mathfont{#2}}{#1}}{#3}}
\newcommand{\MathOptionals}[3][]{\forcsvlist{\DefineCommandOptional{\csname#2\endcsname[##1]}{#1}}{#3}}

\DeclareListParser*{\semiforcsvlist}{;}

\newcommand{\defineB}[3]{\expandafter#1\csname#2\endcsname{#3}}
\newcommand{\ddefine}[4]{\expandafter#1\csname#3\endcsname{\csname#2\endcsname{#4}}}

\makeatletter
\newcommand\newplaincommand[1]{\newplaincom#1\relax}
\def\newplaincom#1,#2\relax{%
\defineB{\newcommand}{#1}{#2}
}
\newcommand\newcommands[2][]{\newcom[#1]#2\relax}
\def\newcom[#1]#2,#3\relax{%
\ddefine{\newcommand}{#1}{#2}{#3}
}
\newcommand\newoperator[2][mathrm]{\newop[#1]#2\relax}
\def\newop[#1]#2,#3\relax{%
\ddefine{\DeclareMathOperator}{#1}{#2}{#3}
}
\makeatother

\newcommand{\DeclareMathOperators}[2][mathrm]{
\semiforcsvlist{\newoperator[#1]}{#2}}

\newcommand{\newtextcommands}[2][text]{
\semiforcsvlist{\newcommands[#1]}{#2}}

\newcommand{\newplaincommands}[1]{
\semiforcsvlist{\newplaincommand}{#1}}

\newcommand{\appendword}[1]{#1}
\newcommand{\appendwords}[1]{\forcsvlist{\appendword}{#1}}

\newcommand{\spaceword}[1]{\text{#1}~}
\newcommand{\spacewords}[2][~]{\forcsvlist{\spaceword}{#2}#1}

\newcommand{\maketextcommand}[4][]{\ddefine{\newcommand}{}{#1\appendwords{#4}}{#2\spacewords[#3]{#4}}}
\newcommand{\maketextcommands}[4][]{\semiforcsvlist{\maketextcommand[#1]{#2}{#3}}{#4}}

\newcommand{\definesentences}[1]{
\maketextcommands[t]{~}{~}{#1}
\maketextcommands[qt]{\quad}{~}{#1}
\maketextcommands[q]{\quad}{\quad}{#1}
}

\makeatletter

\newcommand{\bigmathfrakinner}[2]{\scalebox{1.3}{\(\m@th#1\mathfrak{#2}\)}}
\makeatother

\batchrelax{P}
\MathFonts{mathbb}{A,C,D,K,N,Q,R,Z,F,P,T}
\newcommand{\StandardMathFonts}[1]{%
\MathFonts[bb]{mathbb}{#1}
\MathFonts[bf]{mathbf}{#1}
\MathFonts[bm]{mathbold}{#1}
\MathFonts[c]{mathcal}{#1}
\MathFonts[fr]{mathfrak}{#1}
\MathFonts[rm]{mathrm}{#1}
\MathFonts[sf]{mathsf}{#1}
\MathFonts[scr]{mathscr}{#1}
\MathFonts[cat]{mathsf}{#1}
\MathFonts[field]{mathbb}{#1} 
\MathFonts[functor]{mathscr}{#1} 
}
\StandardMathFonts{A,B,C,D,E,F,G,H,I,J,K,L,M,N,O,P,Q,R,S,T,U,V,W,X,Y,Z}
\newcommand{\StandardMathLowerFonts}[1]{%
\MathFonts[bb]{mathbb}{#1}
\MathFonts[bf]{mathbf}{#1}
\MathFonts[bm]{mathbold}{#1}
\MathFonts[fr]{mathfrak}{#1}
\MathFonts[bigfr]{bigmathfrak}{#1}
\MathFonts[rm]{mathrm}{#1}
\MathFonts[field]{mathbb}{#1} 
}
\StandardMathLowerFonts{a,b,c,d,e,f,g,h,i,j,k,l,m,n,o,p,q,r,s,t,u,v,w,x,y,z}
\MathFonts[bc]{mathbcal}{A,B,C,D,E,F,G,H,I,J,K,L,M,N,O,P,Q,R,S,T,U,V,W,X,Y,Z}
\MathFonts[bscr]{mathscr}{A,B,C,D,E,F,G,H,I,J,K,L,M,N,O,P,Q,R,S,T,U,V,W,X,Y,Z}
\MathFonts[bscr]{mathscr}{k}
\MathSymbols[bf]{mathbold}{alpha,beta,gamma,delta,epsilon,zeta,theta,iota,kappa,lambda,mu,nu,xi,omicron,pi,rho,sigma,tau,upsilon,phi,chi,psi,omega}
\MathSymbols[bf]{mathbold}{Alpha,Beta,Gamma,Delta,Epsilon,Zeta,Theta,Iota,Kappa,Lambda,Mu,Nu,Xi,Omicron,Pi,Rho,Sigma,Tau,Upsilon,Phi,Chi,Psi,Omega}
\MathSymbols[bf]{mathbold}{varepsilon,varrho,varphi,vartheta}
\MathSymbols[bf]{mathbold}{dagger}

\MathSymbols[ul]{underline}{alpha,beta,gamma,delta,epsilon,zeta,theta,iota,kappa,lambda,mu,nu,xi,omicron,pi,rho,sigma,tau,upsilon,phi,chi,psi,omega}
\MathFonts[u]{mathrm}{d}
\MathFonts[rm]{mathrm}{th}

\ColorMath[redbm]{red}{mathbold}{a,b,c,d,e,f,g,h,i,j,k,l,m,n,o,p,q,r,s,t,u,v,w,x,y,z}

\newcommand{\StandardMathOps}[1]{%
\MathOperators[bf]{mathbf}{#1}
\MathOperators[fr]{mathfrak}{#1}
\MathOperators[rm]{mathrm}{#1}
}

\batchrelax{div,Top}
\MathOperators{mathrm}{%
ab,adj,aff,alg,Alt,alt,an,Ann,ant,Ar,Arg,ass,Aut,Av,bimult,%
Bor,Br,%
card,Card,Cay,cd,cdim,Cent,ch,Ch,Char,Cl,cl,codim,Coim,coind,Coker,coker,conj,const,cor,Cor,corad,covol,Covol,%
Det,Diag,diag,diam,Diff,dimproj,disc,dist,Div,Dyn,%
Eq,et,ev,Exp,Ex,Ext,%
Filt,Frac,%
Gal,gen,girth,grad,Grass,%
Hom,Homeo,%
Id,id,im,Image,incl,Ind,ind,Inn,inn,Int,inv,Irr,is,Iso,Isog,%
Jac,%
Ker,Krull,%
lcm,Length,length,Lev,Lip,lox,%
Map,Max,Mor,mor,mult,%
nil,Nm,Norm,norm,Nr,nr,Nrd,Nrm,Nrp,%
Obj,obj,Ob,opp,op,orb,Orb,ord,order,Out,%
Par,park,per,Perm,Pf,Pfaff,PH,Pic,point,pr,Prd,Proj,prox,%
Quot,%
Rad,rad,Range,rank,rdim,red,Reg,reg,res,rk,%
scd,sep,sgn,sing,Sn,soc,Span,spec,spm,Srd,Stab,stab,supp,Supp,support,Syl,sym,symb,%
Tab,Tan,Tgt,Top,Tor,tr,Tr,tra,Trace,Tran,Trans,Trd,trdeg,triv,Trp,twist,%
unit,univ,%
vcd,vol,Vol,%
weight,Wh,%
Zar%
}
\batchrelax{O}
\MathOperators{mathbf}{%
Ad,ad,ann,
Cliff,Cov,CSp,%
depth,Der,der,Discr,Dist,%
End,%
GL,GO,Gr,gr,GSO,GSp,GSpin,GU,%
height,Hess,Hyp,%
Isom,%
Kill,Killing,%
Lie,%
M,%
O,Orth,%
PGL,PGO,PGSO,PGSp,PGU,Pin,PO,Pol,PSL,PSO,PSp,PSU,PU,%
Res,%
SA,Sim,SK,Skew,SL,SO,Sp,Spec,Spin,Spm,SU,SUK,Sym,%
trace,%
U,%
Wt}
\StandardMathOps{%
Ad,ad,Alt,ann,Aut,%
Cliff,CSp,%
depth,Der,der,Dist,%
End,%
GL,GO,Gr,gr,GSO,GSp,GSpin,GU,%
height,Hom,%
Inn,Int,Isom,%
Kill,Killing,%
Lie,%
Mat,%
Orth,%
PGL,PGO,PGSO,PGSp,PGU,Pin,PO,Proj,PSL,PSO,PSp,PSU,PU,%
Res,%
SA,Sim,Skew,SL,SO,Sp,Spec,Spin,Spm,SU,Sym,%
trace,%
}

\batchrelax{sl,sp}
\MathOperators{bigmathfrak}{sl,gl,so,su,sp,pgl,psu,pso}

\makeatletter
\ifx\catstyle\undefined
\newcommand{\catstyle}[1]{\mathsf{#1}} 
\fi
\makeatother

\newcommand{\DefineCategories}[1]{%
\MathOptionals[p]{pcatstyle}{#1}
\MathOptionals[pdash]{pdashcatstyle}{#1}
\MathOptionals[psub]{psubcatstyle}{#1}
}

\batchrelax{Vec,Top}
\DefineCategories{
Ab,Aff,AffSch,Alg,AlgGrp,An,AssAlg,%
BiAlg,%
CoAlg,Cog,Comod,%
Ens,Esg,%
Forget,Funct,%
Grp,grp,%
Hol,Hopf,%
LieAlg,LieGrp,LinAlgGrp,%
Man,Mod,Mon,%
Nat,%
PreSh,%
Rep,Ring,Rings,Rng,Rngs,%
Sch,Sets,Sh,%
Top,Topo,%
Var,Vec%
}

\newcommand{\qqtext}[1]{\quad\text{#1}\quad}
\newcommand{\qtext}[1]{\quad\text{#1}~}

\newcommand{\defineallwords}[1]{
\MathFonts[t]{ttext}{#1}
\MathFonts[q]{qqtext}{#1}
\MathFonts[qt]{qtext}{#1}
}

\defineallwords{adjoint,all,and,any,arbitrary,are,as,%
but,by,%
centralizes,contains,%
divides%
either,elsewhere,equals,even,%
finite,for,from,generates,hence,%
if,If,implies,imply,in,injective,into,is,%
long,%
main,mapping,maximal,minimal,%
nor,normalizes,%
odd,of,on,onto,open,or,otherwise,over,%
real,reductive,representing,represents,resp,respectively,%
semisimple,sending,short,smooth,so,sub,surjective,%
then,thus,to,%
under,uniformly,unless,%
via,%
when,whence,whenever,where,while,with%
}
\defineallwords{avec,%
dans,de,donc,%
et,%
lorsque,%
ou,%
pour,%
si,sinon,sur%
}
\MathFonts[dash]{dashtext}{%
alg,Alg,algebra,algebras,%
comodules,conjugate,%
dimensional,ergodic,%
generic,group,%
LeftMod,linear,%
Mod,module,modules,%
orbit,orbits,%
pair,point,points,%
regular,%
schemes,semilinear,singular,submodule,submodules,subspace%
}
\MathFonts{text}{%
Borel,%
closed,convex,%
dominant,%
ergodic,even,%
hence,homogeneous,%
irreducible,%
maximal,minimal,%
odd,open,otherwise,%
parabolic,%
short,simple%
}

\newcommand{\Set}[1]{\left\{\,#1\,\right\}}
\newcommand{\set}[2][]{\ifblank{#1}{}{\csname#1\endcsname}\{\,#2\,\ifblank{#1}{}{\csname#1\endcsname}\}}
\newcommand{\parenset}[2][]{\csname#1\endcsname(\,#2\,\csname#1\endcsname)}
\newcommand{\angleset}[2][]{\csname#1\endcsname\<\,#2\,\csname#1\endcsname\>}
\newcommand{\st}{\,\middle|\,}
\newcommand{\tst}[1][]{\mathrel{\csname#1\endcsname\vert}}
\newcommand{\sqbrackets}[2][]{\csname#1\endcsname[#2\csname#1\endcsname]}

\newcommand{\Gm}[1][]{\bfG_{\frm\ifblank{#1}{}{,#1}}}

\newplaincommands{Characters,\bmX;Cocharacters,\bmX^{\ast};BorelVariety,\scrB;Mat,\bfM}

\ModifyMathOperator[t]{textop}{sum,prod,bigoplus,bigotimes,coprod,bigcap,bigcup,bigsqcup,bigwedge}

\ModifyMathOperatorStar[s]{textop}{sum,prod,bigoplus,bigotimes,coprod,bigcap,bigcup,bigsqcup,bigwedge}

\newplaincommands{Derived,\cD}

\newcommand{\Commutator}[3][]{\ifblank{#1}{}{\csname#1\endcsname}\lbrack#2\mathbin{,}#3\ifblank{#1}{}{\csname#1\endcsname}\rbrack}

\usepackage{tensor}
\newcommand{\presup}[2]{\prescript{#1}{}{\mkern-2mu#2}}

\newcommand{\<}{\langle}
\renewcommand{\>}{\rangle}

\renewcommand{\epsilon}{\varepsilon}

\DeclareMathOperators[mathbf]{Projective,P;Affine,A}

\DeclareMathOperator{\defeq}{\coloneqq}

\DeclareRobustCommand{\into}{\hookrightarrow}

\newcommand{\qLeftrightarrow}{\quad\Leftrightarrow\quad}

\newcommand{\AlgClosure}[1]{\bar{#1}}

\newcommand{\SepClosure}[1]{{{#1}_{s}}}

\newlength{\arrow}
\newcommand{\WeilRes}[1][]{R\ifblank{#1}{}{_{#1}}}

\newcommand{\PCommutator}[3][]{{\color{blue}\ifblank{#1}{}{\csname#1\endcsname}(}#2\mathbin{\operatorcolorblue{,}}#3{\color{blue}\ifblank{#1}{}{\csname#1\endcsname})}}

\DeclareMathOperators{Graded,Gr;semiregular,sr;nondivisible,nd;nonmultipliable,nm;WeylChamber,WC;Krulldim,Kr.dim;graded,gr;Real,Re;Imaginary,Im;Weyl,W;Specm,Spm}

\MathFonts[redu]{mathrm}{d}

\NewDocumentCommand{\partialderivative}{O{}m>{\SplitArgument{4}{,}} m}{\frac{\ifblank{#2}{\partial\ignorespaces}{\partial^{#2}\kern-0.2ex\ignorespaces}\ignorespaces#1}{\partialderivativeA#3}}
\NewDocumentCommand{\partialderivativeA}{mmmmm}{
\IfNoValueTF{#2}{}{\partial #1}
\IfNoValueTF{#2}{}{\partial #2}
\IfNoValueTF{#3}{}{\partial #3}
\IfNoValueTF{#4}{}{\partial #4}%
\IfNoValueTF{#5}{}{\partial #5}}

\NewDocumentCommand{\partialn}{O{}m>{\SplitArgument{4}{,}} m}{\frac{\ifblank{#2}{\partial}{\partial^{#2}}#1}{\tempB#3}}
\NewDocumentCommand{\tempB}{O{}mmmmm}{\partial #2 \IfNoValueTF{#3}{}{\cdots\partial #3}%
\IfNoValueTF{#4}{}{\partial #4}%
\IfNoValueTF{#5}{}{\partial #5}}

\newtagform{bold}{\bfseries(}{\bfseries)}
\usetagform{bold}
\newcommand{\hleq}[1]{\hlT{\bfseries(#1)}}
\newtagform{hlbold}[\hleq]{}{}

\providecommand{\absvalue}[2][]{\vert#2\vert\ifblank{#1}{}{_{#1}}}

\renewcommand{\sqcup}{\amalg}

\providecommand{\bm}{\mathbold}

\newcommand{\tprime}{\textquotesingle}

\mathcode`l="8000
\begingroup
\makeatletter
\lccode`\~=`\l
\DeclareMathSymbol{\lsb@l}{\mathalpha}{letters}{`l}
\lowercase{\gdef~{\ifnum\the\mathgroup=\m@ne \ell \else \lsb@l \fi}}
\endgroup

\MathFonts[u]{mathurm}{Mor,End,Aut,Hom,Out,Inn}
\batchrelax{mathhl,hlgreen,hlred}

\batchrelax{Im,Re}
\MathOperators{mathrm}{Im,Re}

\DeclareMathOperators{%
Adjoint,Ad;%
SemiSimple,ss;%
Homology,H;
Cohomology,H;
Cocycles,Z;
Coboundaries,B;%
Brauer,Br;%
derived,der;%
simplyconnected,sc;%
adjoint,ad;%
specm,spm;%
Center,Z}

\newtextcommands{forae,for a.e.}

\newcommand{\definesame}[3]{\expandafter#1\csname#3\endcsname{#2}}

\newcommand{\NewMultipleCommands}[2]{%
\forcsvlist{\definesame{\newcommand}{#1}}{#2}
}

\NewMultipleCommands{s}{semisimpleelts,semisimplepart,sspart}
\NewMultipleCommands{u}{unipotentelts,unipotentpart,upart}
\NewMultipleCommands{n}{nilpotentelts,nilpotentpart,nilpart}

\expandafter\patchcmd\csname\string\proof\endcsname
{\normalparindent}{0pt}{}{}

\newtheoremstyle{standard}%
  {5pt}{5pt}
  {} 
  {0pt}{\bfseries}{.}{1em}
  {\thmname{#1}\thmnumber{~#2}\thmnote{~(#3)}\theorembookmark{level=\bookmarktheoremlevel}}

\newtheoremstyle{(swap)}%
  {5pt}{5pt}
  {} 
  {0pt}{\bfseries}{.}{1em}
  {\thmnumber{(#2)\hspace{0.5em}}\thmname{#1}\thmnote{ (#3)}}

\newtheoremstyle{(hlswap)}%
  {5pt}{5pt}
  {} 
  {0pt}{\bfseries}{.}{1em}
  {\thmnumber{\hlT{(#2)}\hspace{0.5em}}\thmname{#1}\thmnote{ (#3)}\theoremswapbookmark{level=\bookmarktheoremlevel}}
  
\newtheoremstyle{(empty)}%
  {5pt}{5pt}
  {} 
  {0pt}{\bfseries}{}{1em}
  {\thmnumber{(#2)}\thmname{\hspace{0.5em}#1}\thmnote{\thmnumber{\hspace{0.5em}}#3.}}

\newtheoremstyle{subsection}%
  {5pt}{5pt}
  {} 
  {18pt}{\bfseries}{.}{0.5em}
  {\thmnumber{#2.}\thmname{\hspace{0.5em}#1}\thmnote{\hspace{0.5em}(#3)}\emptybookmark{level=\bookmarktheoremlevel}}

\newtheoremstyle{subsectionempty}%
  {5pt}{5pt}
  {} 
  {18pt}{\bfseries}{.}{0.5em}
  {\thmnumber{#2.}\thmname{\hspace{0.5em}#1}\thmnote{\hspace{0.5em}#3}\emptybookmark{level=\bookmarktheoremlevel}}

\newtheoremstyle{(hlempty)}%
  {5pt}{5pt}
  {} 
  {0pt}{\bfseries}{}{1em}
  {\thmnumber{\hlT{(#2)}}\thmname{\hspace{0.5em}#1}\thmnote{\thmnumber{\hspace{0.5em}}#3.}\emptybookmark{level=\bookmarktheoremlevel}}

\theoremstyle{(hlempty)}
\providecommand{\maincounter}{subsection}

\theoremstyle{subsectionempty}
\newtheorem{env}[\maincounter]{}

\theoremstyle{subsection}
\newtheorem{theorem}[\maincounter]{Theorem}
\newtheorem{lemma}[\maincounter]{Lemma}
\newtheorem{proposition}[\maincounter]{Proposition}
\newtheorem{corollary}[\maincounter]{Corollary}
\newtheorem{conjecture}[\maincounter]{Conjecture}

\newtheorem{remark}[\maincounter]{Remark}

\theoremstyle{standard}

\theoremstyle{standard}
\newtheorem*{theorem*}{Theorem}
\newtheorem*{lemma*}{Lemma}
\newtheorem*{proposition*}{Proposition}
\newtheorem*{corollary*}{Corollary}
\newtheorem*{definition*}{Definition}
\newtheorem*{definitions*}{Definitions}
\newtheorem*{remark*}{Remark}
\newtheorem*{remarks*}{Remarks}
\newtheorem*{example*}{Example}
\newtheorem*{examples*}{Examples}
\newtheorem*{conjecture*}{Conjecture}
\newtheorem*{exercise*}{Exercise}
\newtheorem*{exercises*}{Exercises}
\newtheorem*{problem*}{Problem}

\theoremstyle{(hlempty)}

\theoremstyle{(empty)}
\newtheorem{starenv}{}

\usepackage[utf8]{inputenc}

\makeatletter
\newdimen\zeroparindent
\def\hlsection{\@starthlsection{section}{1}%
  \z@{.7\linespacing\@plus\linespacing}{.5\linespacing}%
  {\normalfont\bfseries\centering}}
\def\hlsubsection{\@starthlsection{subsection}{2}%
  \zeroparindent{.5\linespacing\@plus.7\linespacing}{-.5em}%
  {\normalfont\bfseries}}
\makeatother

\makeatletter
\def\hlsection{\@starthlsection{section}{2}%
  \zeroparindent{.5\linespacing\@plus.7\linespacing}{-.5em}%
  {\normalfont\bfseries}}
\makeatother

\makeatletter
\def\@hlseccntformat#1{%
  \hlT{\,\protect\textup{\protect\@secnumfont
    \S\csname the#1\endcsname}\,}
    \protect\@secnumnopunct
}

\def\@parenseccntformat#1{%
  \hlT{\,\protect\textup{\protect\@secnumfont
    (\csname the#1\endcsname})\,}
    \protect\@secnumnopunct
}

\def\@hlsect#1#2#3#4#5#6[#7]#8{%
  \edef\@toclevel{\ifnum#2=\@m 0\else\number#2\fi}%
  \ifnum #2>\c@secnumdepth \let\@secnumber\@empty
  \else \@xp\let\@xp\@secnumber\csname the#1\endcsname\fi
  \@tempskipa #5\relax
  \ifnum #2>\c@secnumdepth
    \let\@svsec\@empty
  \else
    \refstepcounter{#1}%
    \edef\@secnumnopunct{%
      \ifdim\@tempskipa>\z@ 
        \@ifnotempty{#8}{\@nx\enspace}%
      \else
        \@ifempty{#8}{}{\@nx\enspace}%
      \fi
    }%
    \protected@edef\@svsec{%
      \ifnum#2<\@m
        \@ifundefined{#1name}{}{%
          \ignorespaces\csname #1name\endcsname\space
        }%
      \fi
      \@parenseccntformat{#1}%
    }%
  \fi
  \ifdim \@tempskipa>\z@ 
    \begingroup #6\relax
    \@hangfrom{\hskip #3\relax\@svsec}{\interlinepenalty\@M #8\par}%
    \endgroup
    \ifnum#2>\@m \else \@tocwrite{#1}{#8}\fi
  \else
  \def\@svsechd{#6\hskip #3\@svsec
    \@ifnotempty{#8}{\ignorespaces#8\unskip
       \@addpunct.}%
    \ifnum#2>\@m \else \@tocwrite{#1}{#8}\fi
  }%
  \fi
  \global\@nobreaktrue
  \@xsect{#5}}

\def\@starthlsection#1#2#3#4#5#6{%
 \if@noskipsec \leavevmode \fi
 \par \@tempskipa #4\relax
 \@afterindenttrue
 \ifdim \@tempskipa <\z@ \@tempskipa -\@tempskipa \@afterindentfalse\fi
 \if@nobreak \everypar{}\else
     \addpenalty\@secpenalty\addvspace\@tempskipa\fi
 \@ifstar{\@dblarg{\@hlsect{#1}{\@m}{#3}{#4}{#5}{#6}}}%
         {\@dblarg{\@hlsect{#1}{#2}{#3}{#4}{#5}{#6}}}%
}

\makeatother

\makeatletter
\newcommand\hlsectionformat[1]{\@hlsectionformat#1\relax}
\def\@hlsectionformat#1\relax{\hlT{#1}}
\makeatother

\providecommand{\bookmarktablelevel}{\bookmarktheoremlevel}
\providecommand{\bookmarkfigurelevel}{\bookmarktheoremlevel}
\providecommand{\bookmarktheoremlevel}{subsection}

\usepackage{bookmark}
\bookmarksetup{open,numbered}
\usepackage{thmtools}
\makeatletter
\pretocmd\endtable{%
  \bookmark[
	level=\bookmarktablelevel,
    dest=\@currentHref,
  ]{Table \thetable: \@currentlabelname}%
}{}{\errmessage{Patching \noexpand\endtable failed}}
\pretocmd\endfigure{%
  \bookmark[
	level=\bookmarkfigurelevel,
    dest=\@currentHref,
  ]{Figure \thefigure: \@currentlabelname}%
}{}{\errmessage{Patching \noexpand\endfigure failed}}
\makeatother

\makeatletter
\newcommand*{\emptybookmark}[1]{%
  \bookmark[
    dest=\@currentHref,
	#1
  ]{%
  \csname the\thmt@envname\endcsname. \thmt@thmname\space
    \ifx\@currentlabelname\@empty
    \else
      \space\@currentlabelname%
    \fi
  }%
}
\newcommand*{\emptybookmarkstar}[1]{%
  \bookmark[
    dest=\@currentHref,
	#1
  ]{
  \@currentlabelname%
  }%
}
\newcommand*{\theorembookmark}[1]{%
  \bookmark[dest=\@currentHref,#1]{%
  \thmt@thmname\space\csname the\thmt@envname\endcsname
    \ifx\@currentlabelname\@empty
    \else
      \space(\@currentlabelname)%
    \fi
  }%
}

\newcommand*{\theoremswapbookmark}[1]{%
  \bookmark[
    dest=\@currentHref,
	#1
  ]{%
  (\csname the\thmt@envname\endcsname) \thmt@thmname\space
    \ifx\@currentlabelname\@empty
    \else
      \space(\@currentlabelname)%
    \fi
  }%
}
\newcommand*{\theoremswapbookmarkstar}[1]{%
  \bookmark[
    dest=\@currentHref,
	#1
  ]{%
  \thmt@thmname\space
    \ifx\@currentlabelname\@empty
    \else
      \space(\@currentlabelname)%
    \fi
  }%
}
\makeatother

\usepackage{enumitem}

\newcommand{\symbolitem}[1]{\global\itemsymboltrue\item}

\newcommand{\perhapssymbol}[1]{%
  \ifitemsymbol#1\global\itemsymbolfalse\fi
}
\newif\ifitemsymbol

\newcommand{\defaultsymbol}{\tprime}

\SetEnumitemKey{bf}{label=\textbf{(\arabic*\protect\perhapssymbol{\defaultsymbol})},leftmargin=*,topsep=2pt,itemsep=2pt}
\SetEnumitemKey{bfflush}{label=\textbf{(\arabic*\protect\perhapssymbol{\defaultsymbol})},wide,labelindent=0pt,topsep=2pt,itemsep=2pt,}
\SetEnumitemKey{bfroman}{label=\textbf{(\roman*\protect\perhapssymbol{\defaultsymbol})},leftmargin=*,topsep=2pt,itemsep=2pt,}
\SetEnumitemKey{bfromanflush}{label=\textbf{(\roman*\protect\perhapssymbol{\defaultsymbol})},wide,labelindent=0pt,topsep=2pt,itemsep=2pt,}
\SetEnumitemKey{bfRoman}{label=\textbf{(\Roman*\protect\perhapssymbol{\defaultsymbol})},leftmargin=*,topsep=2pt,itemsep=2pt,}
\SetEnumitemKey{bfAlpha}{label=\textbf{(\Alph*\protect\perhapssymbol{\defaultsymbol})},leftmargin=*,topsep=2pt,itemsep=2pt,}
\SetEnumitemKey{bfalpha}{label=\textbf{(\alph*\protect\perhapssymbol{\defaultsymbol})},leftmargin=*,topsep=2pt,itemsep=2pt,}
\SetEnumitemKey{bfalphaflush}{label=\textbf{(\alph*\protect\perhapssymbol{\defaultsymbol})},wide,labelindent=0pt,topsep=2pt,itemsep=2pt,}
\SetEnumitemKey{1)}{label=\textbf{\arabic*}),leftmargin=*}
\SetEnumitemKey{bfromanprime}{label=(\textbf{\roman*}\tprime),leftmargin=*}
\SetEnumitemKey{bfalphaprime}{label=(\textbf{\alph*}\tprime),leftmargin=*}
\setlist[enumerate]{bf}

\definesentences{for,all;such,that}

\newcommand{\sqIndex}[3][]{\csname#1\endcsname[#2:#3\csname#1\endcsname]}
\colorlinks

\DeclareMathOperator{\Mahler}{M}
\newcommand{\mahler}{\underline{m}}
\newcommand{\Weilheight}{\underline{h}}

\theoremstyle{standard}
\newtheorem{theoremA}{Theorem}

\crefname{theoremA}{Theorem}{Theorems}

\title{Bottom of the Length Spectrum of Arithmetic Orbifolds}

\author{Mikołaj Frączyk}
\address{Department of Mathematics \\ University of Chicago \\ Chicago, IL 60637 \\ USA}
\email{mfraczyk@math.uchicago.edu}

\author{Lam L. Pham}
\address{Department of Mathematics, Brandeis University, 415 South Street, Waltham, MA 02453, USA \& The Einstein Institute of Mathematics, Edmond J. Safra Campus, Givat Ram, The Hebrew University of Jerusalem Jerusalem, 91904, Israel}
\email{lampham@brandeis.edu}
\thanks{L. P. acknowledges the support of the Zuckerman STEM Leadership Program (as Zuckerman Postdoctoral Scholar)}

\newcommand{\AC}{\AlgClosure}
\newcommand{\AGal}[1]{\Gal(\AC{#1}/#1)}

\DeclareMathOperator{\elliptic}{(e)}
\DeclareMathOperator{\hyp}{(h)}
\renewcommand{\hyp}{\text{ntor}}
\renewcommand{\elliptic}{\text{tor}}
\begin{document}

\begin{abstract}
We prove that cocompact arithmetic lattices in a simple Lie group are uniformly discrete if and only if the Salem numbers are uniformly bounded away from \(1\). We also prove an analogous result for semisimple Lie groups. Finally, we shed some light on the structure of the bottom of the length spectrum of an arithmetic orbifold \(\Gamma\backslash X\) by showing the existence of a positive constant \(\delta(X)>0\) such that squares of lengths of closed geodesics shorter than \(\delta\) must be pairwise linearly dependent over \(\Q\).
\end{abstract}

\maketitle
\tableofcontents

\section{Introduction and main results}

\setcounter{subsection}{-1}

Let \(G\) be a semisimple Lie group, \(X\) its associated symmetric space. We study orbifolds \(\Gamma\backslash X\) where \(\Gamma\subset G\) is a cocompact arithmetic lattice. We investigate the uniform discreteness of cocompact arithmetic lattices in a given semisimple Lie group \(G\). This was first conjectured by Margulis for higher rank semisimple Lie groups \cite[Chapter IX, \S 4.21 (B)]{Margulis1991}:

\begin{conjecture}
\label{Margulis-conjecture}
Let \(G\) be a connected semisimple \(\R\)-group. Suppose \(\rk_{\R}(G)\geq 2\). Then, there exists a neighborhood \(U\subset G(\R)\) of the identity such that for any irreducible cocompact lattice \(\Gamma \subset G(\R)\), the intersection \(U\cap\Gamma\) consists of elements of finite order.
\end{conjecture}

Margulis' Arithmeticity Theorem \cite{Margulis1975} implies that irreducible lattices in higher-rank semisimple Lie groups are arithmetic. Although in rank one, non-arithmetic lattices do exist, restricting to \emph{arithmetic lattices}, we expect the same statement to hold. Therefore, \cref{Margulis-conjecture} may be extended to rank 1 groups by adding the assumption of arithmeticity.

\begin{conjecture}[Margulis' conjecture]
\label{Margulis-conjecture'}
Let \(G\) be a connected semisimple \(\R\)-group. Then, there exists a neighborhood \(U\subset G(\R)\) of the identity such that for any irreducible \emph{arithmetic} cocompact lattice \(\Gamma \subset G(\R)\), the intersection \(U\cap\Gamma\) consists of elements of finite order.
\end{conjecture}

Henceforth, this conjecture will be referred to as \emph{Margulis' Conjecture}. In \(\SL_{2}(\R)\), the conjecture takes the following form:

\begin{conjecture}[Short Geodesic Conjecture]
\label{Sury-conjecture}
There exists a neighborhood \(U\subset \SL_{2}(\R)\) of the identity such that for any arithmetic torsion-free cocompact lattice \(\Gamma\subset \SL_{2}(\R)\), we have \(\Gamma\cap U=\{e\}\).
\end{conjecture}

It turns out that this conjecture is equivalent to a conjecture about algebraic integers known as Salem numbers. A \emph{Salem number} is a real algebraic integer \(x>1\) which is conjugate to \(x^{-1}\) and all other Galois conjugates lie on the unit circle. These numbers appear naturally as eigenvalues of hyperbolic elements of arithmetic lattices in \(\SL_{2}(\R)\). Using this connection, Chinburg, Neumann and Reid \cite{NeumannReid1992} and Sury \cite{Sury1992} showed that \cref{Sury-conjecture} is equivalent to:

\begin{conjecture}[Salem conjecture]
\label{Salem-conjecture}
There exists \(\epsilon>0\) such that \(x>1+\epsilon\) for every Salem number \(x\).
\end{conjecture}

In higher rank semisimple Lie groups, eigenvalues of hyperbolic elements may be arbitrary algebraic units. Therefore, it may come as a surprise that for \emph{simple} Lie groups, \cref{Margulis-conjecture'} is still equivalent to \cref{Salem-conjecture}:

\begin{theoremA}
\label{theorem:A}
Let \(G\) be a simple, connected Lie group. Then, the \hyperref[Margulis-conjecture']{Margulis conjecture} (\ref{Margulis-conjecture'}) for \(G\) is equivalent to the \hyperref[Salem-conjecture]{Salem conjecture} (\ref{Salem-conjecture}).
\end{theoremA}

The main new content of \cref{theorem:A} is that the Salem conjecture implies the Margulis conjecture for any simple Lie group \(G\). The fact that \cref{Salem-conjecture} follows from \cref{Margulis-conjecture'} for any isotropic almost simple \(\R\)-group already follows from the main theorem of the second-named author and F. Thilmany \cite{PhamThilmany2021}.

Let us also mention that when \(G\) is the isometry group of real hyperbolic \(n\)-space and \(\Gamma\) an arithmetic lattice of the simplest type \cite{Vinberg1993a}, a version of \cref{theorem:A} follows from \cite{EmeryRatcliffeTschantz2019} where the authors establish a correspondence between translation lengths and Salem numbers for \(\SO(n,1)\). Our approach, however, is different: since such a correspondence cannot be expected to hold in higher rank, we instead focus on \emph{short} geodesics. In this case, we construct Salem numbers using the Galois action on the root system and certain unexpected multiplicative relations between eigenvalues of short elements. A key ingredient in our argument is a theorem of Amoroso and David \cite{AmorosoDavid1999} on a higher dimensional analogue of Lehmer's conjecture. Although bounds towards the Lehmer conjecture have long been known to have applications to systoles of arithmetic orbifolds \cite{Margulis1991}%
\footnote{Recently, Dobrowolski's theorem \cite{Dobrowolski1979} was applied in  \cite{Fraczyk2021}, \cite{Belolipetsky2021}, \cite{LapanLinowitzMeyer2022}, \cite{FrazcykHurtadoRaimbault2022}.},
to our knowledge, this is the first time that bounds in the higher dimensional context are applied to study the length spectrum.



To state our most general result, of which \cref{theorem:A} is a direct corollary, we need to recall the relevant notions of Mahler measure and Lehmer's conjecture.

Let \(P\in\C[X]\) be a polynomial of degree \(d\) with leading coefficient \(a_{d}\) and roots \(\alpha_{1},\hdots,\alpha_{d} \in \C\). The \emph{Mahler measure} of \(P\) is defined as
\[
\Mahler(P)=\absvalue{a_{d}}\prod_{i=1}^{d} \max\set[big]{1,\absvalue{\alpha_{i}}}.
\]
We may relabel the roots as \(\alpha_{1}, \hdots,\alpha_{s(P)}, \hdots, \alpha_{d}\), where \(s(P)\) denotes the number of roots of \(P\) in \(\C\) of modulus \(>1\). If \(\alpha \in \AC{\Q}\) is an algebraic integer, the Mahler measure \(\Mahler(\alpha)\) of \(\alpha\) is by definition the Mahler measure \(\Mahler(P_{\alpha})\) of its minimal polynomial \(P_{\alpha}\) over \(\Z\) (which is not monic, but whose coefficients are coprime) \cite{BombieriGubler2006}. The Mahler measure of algebraic integers is invariant under the action of \(\Gal(\AC{\Q} / \Q)\).

By Kronecker's theorem, if \(P\) is an integral monic irreducible polynomial,
\[
\Mahler(P)=1
\quad\text{if and only if}\quad
\begin{cases}
P(X)=X,\tor\\
P~\text{is a cyclotomic polynomial.}
\end{cases}
\]

A conjecture attributed to Lehmer \cite{Lehmer1933} states that we have a strict dichotomy regarding the Mahler measure of irreducible monic polynomials.

\begin{conjecture}[Lehmer]
\label{Lehmer-conjecture}
There exists \(\epsilon>0\) such that for any (irreducible) monic polynomial \(P\) with integer coefficients, either
\[
\Mahler(P)=1 \quad \text{or} \quad \Mahler(P)>1+\epsilon.
\]
\end{conjecture}

By allowing the bound to depend on the number of roots outside the unit circle, we obtain the following weaker version of Lehmer's conjecture, first put forth by Margulis \cite{Margulis1991} in connection with \cref{Margulis-conjecture}.

\begin{conjecture}[weak Lehmer]
\label{weak-Lehmer-conjecture}
For each \(s \in \N\), there exists \(\epsilon(s)>0\) such that for any (irreducible) monic polynomial \(P\) with integer coefficients and \(s(P) \leq s\), either
\[
\Mahler(P) = 1 \qor \Mahler(P)>1+\epsilon(s).
\]
\end{conjecture}

Margulis \cite{Margulis1991} observed that \cref{weak-Lehmer-conjecture} applied to the eigenvalues of semisimple elements of \(\Gamma\) would readily imply \cref{Margulis-conjecture}. Algebraic integers with \(s=1\) are precisely \emph{Salem numbers}.

It is possible to go in the other direction as well: \cref{Margulis-conjecture'} for certain families of semisimple groups implies \cref{weak-Lehmer-conjecture}.

\begin{theorem}[Pham-Thilmany\cite{PhamThilmany2021}]
\label{main-theorem}
Fix an absolutely (almost) simple isotropic \(\R\)-group \(\bfF\), and consider, for each integer \(r \geq 1\), the family of semisimple \(\R\)-groups
\[
\scrT_{\bfF}^{(r)} = \Set{\prod_{i=1}^{r_1} \bfF \times \prod_{i=1}^{r_2} \Res_{\C/\R}(\bfF) \st r_1, r_2 \in \N,~r_1+2r_2 \leq r}.
\]
Then, for any \(r\geq 1\), \cref{Margulis-conjecture'} for any of the families \(\scrT_{\bfF}^{(r)}\), implies \cref{weak-Lehmer-conjecture} (with \(s(P)\leq r\)).
\end{theorem}

Similarly to \cref{theorem:A}, we can actually prove that the weak Lehmer conjecture with fixed \(s(P)\) will imply \cref{Margulis-conjecture'} for a quite large family of semisimple groups, subject only to certain conditions on the number of simple factors. To be more precise, we will consider a version of weak Lehmer conjecture for algebraic numbers with prescribed number of real and non-real conjugates outside the unit circle.
Let \(r_{1},r_{2}\in \N\). Say that an algebraic integer \(\alpha\) \emph{has signature \((r_{1},r_{2})\)} if
\begin{align*}
r_{1}&=\card\Set{\sigma\in \Hom(\Q(\alpha),\C)\st \sigma(\alpha)\in \R \tand \absvalue{\alpha}> 1},\\
2r_{2}&=\card\Set{\sigma\in \Hom(\Q(\alpha),\C)\st \sigma(\alpha)\in \C\setminus \R, \tand \absvalue{\alpha}> 1}.
\end{align*}
Then, the \emph{Lehmer conjecture of signature \((r_{1},r_{2})\)} is the following statement.

\begin{env}[Lehmer conjecture of signature \((r_{1},r_{2})\)]
\label{Lehmer-conjecture-signature}
There exists \(\varepsilon>0\) such that every algebraic integer \(\alpha\) of signature \((r_{1},r_{2})\) satisfies \(\Mahler(\alpha)\geq \varepsilon\).
\end{env}

We now define the corresponding notion of signature for a real semisimple Lie group. Let \(G\) be a semisimple, simply connected real algebraic group. We can decompose \(G\) as a product of simple real Lie groups and a certain number of almost simple real Lie groups, which are restrictions of scalars of complex simple groups. Let \(r_{1}(G)\) be the number of non-compact real simple factors and \(r_{2}(G)\) be the number of complex simple factors. For example the group
\[
G=\SL_{4}(\R)\times \SU(2,2)\times \SL_{4}(\C)^{2}\times \SU(4)^{3}
\]
has signature \(r_{1}(G)=2\) and \(r_{2}(G)=2\). We will refer to the pair \((r_{1}(G),r_{2}(G))\) as the \emph{signature} of \(G\). The general version of \cref{theorem:A} is

\begin{theoremA}
\label{theorem:B}
Let \(G\) be a semisimple, simply connected real algebraic group of signature \((r_{1},r_{2})\). Then, the \hyperref[Margulis-conjecture']{Margulis conjecture} for \(G\) follows from the \hyperref[Lehmer-conjecture-signature]{Lehmer conjecture} (\ref{Lehmer-conjecture-signature}) for all signatures \((s,t)\) where \(0\leq t\leq r_{2}\) and \(1\leq s+2t\leq r_{1}+2r_{2}\).
\end{theoremA}

Note that in the other direction, the main theorem of \cite{PhamThilmany2021} shows that the \hyperref[Margulis-conjecture']{Margulis conjecture} for \(G\) for all signatures \((s,t)\) where \(0\leq t\leq r_{2}\) and \(1\leq s+2t\leq r_{1}+2r_{2}\) implies the \hyperref[Lehmer-conjecture-signature]{Lehmer conjecture} (\ref{Lehmer-conjecture-signature}) for the same set of signatures.

In addition to the previous results, our method of proof also enables us to draw some conclusions about the length spectra of arithmetic orbifolds, independently of the validity of the Lehmer conjecture(s). We recall that the \emph{length spectrum} of such an orbifold \(\Gamma\backslash X\) is the set of lengths of closed geodesics. In \(\Gamma\backslash X\), the length of the closed geodesic corresponding to a semisimple element \(\gamma\in\Gamma\subset\bfG(\fieldk)\) is given by the formula%
\footnote{Refer to \cref{section:2} for the description of \(\fieldk\), \(\bfG\), \(\bfT\), and \cref{subsection:notation} below for \(V_{\fieldk}\).}
(see \cite[\S 8]{PrasadRapinchuk2009})
\begin{equation}
\label{eq:length}
\ell(\gamma)=\left(\sum_{\sigma\in V_{\fieldk}}\sum_{\lambda\in \Phi(\bfG,\bfT)}\big(\log\absvalue{\lambda(\gamma)}_{\sigma}\big)^{2}\right)^{1/2},
\end{equation}
The length spectrum is the set
\[
\cL(\Gamma\backslash X)=\{\ell(\gamma)\mid \gamma\in \Gamma ~\text{semisimple}\}.
\]
\hyperref[Margulis-conjecture']{Margulis' conjecture} can be restated by saying that there exists a constant \(\delta>0\), depending only on \(X\), such that \(\cL(\Gamma\backslash X)\cap (0,\delta)=\emptyset\) for all arithmetic lattices \(\Gamma\). We can now state:

\begin{theoremA}
\label{theorem:C}
Let \(G\) be a semisimple Lie group without compact factors. There exists \(\delta>0\) with the following property. For any irreducible arithmetic lattice \(\Gamma\subset G\) the intersection
\[
\{\ell^{2}\mid \ell\in \cL(\Gamma\backslash X)\}\cap [0,\delta)
\]
is contained in a cyclic subgroup of \(\R\).
In particular, if \(\ell(\gamma_{1})\) and \(\ell(\gamma_{2})\) are two lengths whose squares are independent over \(\Q\), then they cannot be simultaneously small.
\end{theoremA}

We refer the reader to \cref{section:5} where a more precise result is proved. We only note that the value of \(\delta\) can be made to grow with the degree of the trace field of \(\Gamma\).

Finally, we obtain a precise rank-one obstruction to \hyperref[Margulis-conjecture']{Margulis' conjecture}: if a semisimple element \(\gamma\) generates a higher rank torus, then the length of the corresponding closed geodesic admits a uniform lower bound.

\begin{theoremA}
\label{theorem:D}
Let \(G\) be a noncompact simple (real or complex) Lie group. Let \(\gamma\) be a semisimple element of an arithmetic lattice \(\Gamma\subset G\) with trace field \(\fieldk\). Let \(S\) be the connected component of the the Zariski closure of the subgroup generated by \(\gamma\). If \(\rk S\geq 2\) then
\[
\ell(\gamma)\geq C_{1}^{1/2}\cdot \sqIndex{\fieldk}{\Q}^{1/2}\big(\log \sqIndex{\fieldk}{\Q}\big)^{-C_{2}/2}.
\]
\end{theoremA}

\cref{theorem:D} will also be proved in \cref{section:5}. Just as for \cref{theorem:C}, the lower bound improves as the degree of the trace field grows.

\subsection{Some ideas of the proofs}

The proof of \cref{theorem:B} relies on lower bounds on heights of tuples of multiplicatively independent algebraic units due to Amoroso and David \cite{AmorosoDavid1999}. Let \(\fieldk\) be the trace field of \(\Gamma\) and let \(\gamma\in \Gamma\) be a non-torsion semisimple conjugacy class intersecting an explicit small neighbourhood \(U\) of identity in \(G\). The bounds in \cite{AmorosoDavid1999} are used to show that the eigenvalues of \(\gamma\) are multiplicatively dependent. The Galois group of \(\fieldk\) acts on this set of eigenvalues, so their multiplicative dependence can be leveraged to construct a \(\Q^\times\)-valued Galois cocycle over this action. We prove that this cocycle takes values in the torsion subgroup of \(\Q^{\times}\), which is \(\{\pm 1\}\), and use it to construct an explicit algebraic number \(\tilde{\alpha}\), expressed as a product of eigenvalues of \(\gamma\). The signature of \(\tilde{\alpha}\) lies in the range \(\set{(s,t)\tst 0\leq t\leq r_{2}\tand 1\leq s+2t\leq r_{1}+2r_{2}}\), and its Mahler measure can be made arbitrarily small, provided that \(U\) is small enough. It transpires that the failure of \hyperref[Margulis-conjecture']{Margulis conjecture} for \(G\) would imply the failure of Lehmer's conjecture for at least one signature \((s,t)\) in the specified range.

\subsection{Notations}
\label{subsection:notation}
Throughout the paper,
\(\fieldk\) stands for a number field, \(\bfG\) is a semisimple \(\fieldk\)-group and \(\bfT\) is a \(\fieldk\)-torus containing a semisimple element \(\gamma\in \bfG(\fieldk)\). 
The set of valuations of \(\fieldk\) is denoted \(V_{\fieldk}\), and for any embedding \(\sigma:\fieldk\into\C\), we denote by \(\fieldk_{\sigma}\) the corresponding completion.
For any Galois extension \(\fieldk'/\fieldk\), its Galois group is denoted \(\Gal(\fieldk'/\fieldk)\).
We occasionally use Vinogradov's notation \(f\ll g\), which means that \(f\) is bounded by a constant times \(g\).
\(\N^{*}\) stands for the set of positive natural numbers.
All logarithms are natural. 
\subsection*{Acknowledgments}

Both authors would like to thank Emmanuel Breuillard for pointing them to the paper of Amoroso and David and suggesting that \cref{theorem:C} may hold.
The second named author would also like to thank Elon Lindenstrauss for many  helpful conversations and interesting questions related to the results of this paper. 






\section{Algebraic Preliminaries}
\label{section:2}

In this section, we gather some relevant facts about algebraic groups, arithmetic lattices, and Diophantine geometry. For the remainder of the paper, \(G\) will be the real (semi)simple Lie group without compact factors appearing in \crefrange{theorem:A}{theorem:D}. We will write \(X\) for the symmetric space associated to \(G\), equipped with the canonical \(G\)-invariant Riemannian metric induced by the Killing form on its Lie algebra. The meaning of \(\Gamma\), \(\fieldk\), \(\bfG\), \(\bfT\) established below will be fixed throughout the paper.

\subsection{Arithmetic lattices}

Let \(\fieldk\) be a number field and let \(\bfG\) be a simply connected, semisimple algebraic \(\fieldk\)-group such that \(\bfG(\fieldk\otimes \R)\simeq G\times G'\), where \(G'\) is compact. Note that this means that \(\dim \bfG\) is uniformly bounded in terms of \(\dim G\). Let \(\Gamma\subset \bfG(\fieldk)\) be an \emph{arithmetic subgroup}, by which we mean a subgroup such that \(\rho(\Gamma)\) is commensurable with \(\rho(\bfG(\fieldk))\cap \GL_{m}(\cO_{\fieldk})\), where \(\rho:\bfG\into\GL_{m}\) is a representation  defined over \(\fieldk\) with finite kernel. The commensurability class of \(\Gamma\) does not depend on the chosen representation.

The group \(\Gamma\), diagonally embedded in \(\bfG(\fieldk\otimes\R)\) using the inequivalent archimedean places of \(\fieldk\), projects to a lattice in \(G\) which we shall denote by the same letter. Up to commensurability, all arithmetic lattices in \(G\) arise in this way. The field \(\fieldk\) is called the \emph{trace field} of \(\Gamma\). Finally, we remark that \(\Gamma\) is irreducible if \(\bfG\) is a simple \(\fieldk\)-group. For more details on arithmetic lattices, we refer the reader to \cite{Borel1969,Margulis1991,PlatonovRapinchuk1994}.\nocite{Borel2019}

\subsection{Characteristic polynomials}

From now on let \(\rho=\Ad\) be the adjoint representation of \(\bfG\). For any \(g\in \bfG(\fieldk)\), let \(f_{g}\) be the characteristic polynomial of \(\Ad(g)\), and assuming that \(g\) is semisimple, let \(\bfT\) be a maximal \(\fieldk\)-torus containing \(g\). We denote by \(\Characters(\bfT)\) the group of characters
\[
\Characters(\bfT)=\Hom\big(\bfT_{\AC{\fieldk}},\Gm[\AC{\fieldk}]\big),
\]
where \(\Gm\) denotes the multiplicative group over \(\fieldk\), and \(\AC{\fieldk}\) the algebraic closure of \(\fieldk\).

Let \(\Phi(\bfG,\bfT)\) denote the absolute root system of \(\bfG\) relative to \(\bfT\). If \(g\in \bfT(\AC{\fieldk})\), the polynomial \(f_{g}(t)\) is
\[
f_{g}(t)= t^{\rk\bfG} \prod_{\chi\in\Phi(\bfG,\bfT)} (t-\chi(g))^{m_{\chi}}\qtforall
t\in \bfT(\AC{\fieldk}),
\]
where \(\set{m_{\chi}\in\N^{\ast}\tst \chi\in\Omega}\) are the dimensions of \(\chi\)-isotypic subspaces of \(\frg\). The root system generates a finite-index subgroup of \(\Characters(\bfT)\).




\subsection{Splitting fields of tori}

We recall some facts about splitting fields of tori (for more details, see, e.g., \cite[\S 1]{JouveKowalskiZywina2013}).

We keep the previous notation. The natural left action of \(\Gal(\AC{\fieldk}/\fieldk)\) on \(\Characters(\bfT)\) is given by
\begin{equation}
\label{eq:2.3}
\sigma\big(\chi(t)\big)=(\presup{\sigma}{\chi})\big(\sigma(t)\big),\qtfor \sigma \in\Gal(\AC{\fieldk}/\fieldk),~\chi\in\Characters(\bfT),~t\in \bfT(\AC{\fieldk}).
\end{equation}
Equivalently,
\begin{equation*}
\label{eq:2.4}
\presup{\sigma}{\chi}(t)=\sigma\big(\chi\big(\sigma^{-1}(t)\big)\big),\qtfor \sigma \in\Gal(\AC{\fieldk}/\fieldk),~\chi\in\Characters(\bfT),~t\in \bfT(\AC{\fieldk}).
\end{equation*}

Let \(\fieldk_{\bfT}\) denote the \emph{splitting field} of \(\bfT\), that is, the minimal extension \(\fieldK\subset \AC{\fieldk}\) such that \(\bfT_{\fieldK}\) is split, i.e., \(\bfT_{\fieldK}\simeq (\Gm[\fieldK])^{r}\) for some \(r\in\N^{\ast}\). Equivalently, this is the minimal extension \(\fieldK\) of \(\fieldk\) for which \(\Gal(\AC{\fieldk}/\fieldK)\) acts trivially on \(\Characters(\bfT)\); this is a Galois extension of \(\fieldk\).
We have a representation \(\AGal{\fieldk}\to \Aut(\Characters(\bfT))\) which factors through an injective homomorphism \(\Gal(\fieldk_{\bfT}/\fieldk)\into \Aut(\Characters(\bfT))\). As a result, the Galois group
\[
\scrG_{\bfT}\defeq\Gal(\fieldk_{\bfT}/\fieldk)
\]
acts on the (absolute) root system \(\Phi=\Phi(\bfG,\bfT)\) by automorphisms, and hence, \(\scrG_{\bfT}\) may be identified with a subgroup of \(\Aut(\Phi)\). It follows that
\begin{equation*}
\absvalue{\scrG_{\bfT}}\leq \absvalue{\Aut(\Phi)}\leq (\dim\bfG-\rk\bfG)!.
\end{equation*}
%
%

\begin{remark*}
The degree of the splitting field of \(\bfT\) is also equal to the order of the \emph{splitting group} \cite{Voskresenskii1998}, a subgroup of \(\GL_{d}(\Z)\), where \(d=\dim(\bfT)\), and its order is bounded in terms of \(\dim \bfG\) only \cite{UllmoYafaev2014,Friedland1997}. So for any \(d\in\N\), there exists a constant \(\beta(d)\in\N\) such that
\[
\absvalue{\scrG_{\bfT}}=
[\fieldk_{\bfT}:\fieldk]\leq \beta(d),\qtif \rk(\bfG)\leq d.
\]
\end{remark*}

\subsection{Splitting fields of elements}

Let \(f_{g}\) denote the characteristic polynomial of \(\Ad(g)\in\GL_{n}(\fieldk)\).
We denote the splitting field of \(f_{g}\) over \(\fieldk\) by \(\fieldk_{g}\). If \(g_{\sspart}\) denotes the semisimple part of \(g\), then \(\fieldk_{g}=\fieldk_{g_{\sspart}}\) \cite[\S 2.3]{JouveKowalskiZywina2013}. In addition,
\begin{equation}
\fieldk_{g}=\fieldk\big(\set{\chi(g) \tst \chi\in\Phi(\bfG,\bfT)}\big).
\end{equation}

\begin{lemma*}
For any semisimple \(g\in \bfG(\fieldk)\), we have \(\fieldk_{\bfT}\supseteq\fieldk_{g}\) and thus, \([\fieldk_{g}:\fieldk]\leq [\fieldk_{\bfT}:\fieldk]\).
\end{lemma*}

\begin{proof}
For any \(\sigma\in\AGal{\fieldk}\), and \(\chi\in\Phi(\bfG,\bfT)\), we have by \eqref{eq:2.3} that
\begin{equation}
\label{eq:2.7}
\sigma\big(\chi(g)\big)=\presup{\sigma}{\chi}\big(\sigma(g)\big)=\presup{\sigma}{\chi}(g),
\end{equation}
since \(g\in \bfG(\fieldk)\). In particular, since \(\Gal(\AC{\fieldk}/\fieldk_{\bfT})\) acts trivially on \(\Characters(\bfT)\), this implies that \(\presup{\sigma}{\chi}=\chi\) for all \(\sigma\in\Gal(\AC{\fieldk}/\fieldk_{\bfT})\), so
\[
\sigma(\chi(g))=\chi(g)\qtforall \sigma\in \Gal(\AC{\fieldk}/\fieldk_{\bfT}),
\]
or, equivalently, that \(\sigma(u)=u\) for all \(u\in \set{\chi(g) \tst \chi\in\Phi(\bfG,\bfT)}\), so it follows that \(\fieldk_{\bfT}\supseteq\fieldk_{g}\).
\end{proof}

\subsection{Mahler measure and heights of elements in \(\Gamma\)}

We now let \(\gamma\in \Gamma\) be semisimple. Recall that we are working with the adjoint representation \(\Ad\) of \(\bfG\) on its Lie algebra \(\frg\). Let \(\bfT\) be a maximal \(\fieldk\)-torus containing \(\gamma\) and let \(\Phi=\Phi(\bfG,\bfT)\) be the corresponding absolute root system.

Since \(\Gamma\) is arithmetic, the eigenvalues of \(\Ad\gamma\) are algebraic integers, and in particular,
\[
\tag{\(\star\)}
\chi(\gamma)\in \cO_{\fieldk_{\gamma}}
\qtext{for any} \chi\in\Phi.
\]
We now define the \emph{Mahler measure} of \(\gamma\) to be
\begin{equation}
\Mahler(\gamma)\defeq \prod_{\chi\in\Phi} \Mahler\big(\chi(\gamma)\big),
\end{equation}
where the Mahler measure of \(\chi(\gamma)\) is the Mahler measure of its minimal polynomial over \(\Z\). The \emph{logarithmic Mahler measure} of \(\gamma\) is
\begin{equation}
\mahler(\gamma)\defeq\log\Mahler(\gamma)=\sum_{\chi\in\Phi} \mahler\big(\chi(\gamma)\big).
\end{equation}
We also define the \emph{logarithmic height} of \(\gamma\) to be
\begin{equation}
\Weilheight(\gamma)\defeq \sum_{\chi\in\Phi} \Weilheight\big(\chi(\gamma)\big),
\end{equation}
where, for an algebraic number \(\alpha\), we denote by \(\Weilheight(\alpha)\) its \emph{absolute logarithmic Weil height} \cite{BombieriGubler2006}.

\begin{lemma*}
Let \(\gamma\in \bfG(\cO_{\fieldk})\). Let \(f_{\gamma}(X)\) be the characteristic polynomial of \(\Ad\gamma\). Then,
\[
\mahler(f_{\gamma})\leq \mahler(\gamma)\qand
\Weilheight(\gamma)\leq \mahler(\gamma) \leq [\fieldk_{\gamma}:\Q]\cdot \Weilheight(\gamma).
\]
\end{lemma*}

\begin{proof}
Let \(\alpha\in\AC{\Q}\), with minimal polynomial \(f_{\alpha}\) over \(\Z\) (so if \(\alpha\) is not an algebraic integer, \(f_{\alpha}\) is not necessarily monic, but if \(\alpha\) is an algebraic integer, it is). Then, the Mahler measure and the height of \(\alpha\) are related by the following identity \cite{BombieriGubler2006}:
\begin{equation}
\mahler(\alpha)=\mahler(f_{\alpha})=[\Q(\alpha):\Q]\cdot \Weilheight(\alpha).
\end{equation}
On the other hand, for each algebraic number \(\alpha\), we have
\[
\max\{1,\absvalue{\alpha}\}=\Mahler((t-\alpha))\leq \Mahler(\alpha).
\]
With \(\log^{+}(x)\defeq \max\{0,\log(x)\}\), we get \(\log^{+}|\alpha|\leq \mahler(\alpha)\). Since each \(\chi(\gamma)\) is an algebraic integer, we have for each \(\chi(\gamma)\),
\[
\mahler\big(\chi(\gamma)\big)=[\Q\big(\chi(\gamma)\big):\Q]\cdot \Weilheight\big(\chi(\gamma)\big)
\]
so
\begin{equation}
\mahler\big(\chi(\gamma)\big)\leq [\fieldk_{\gamma}:\Q]\cdot \Weilheight\big(\chi(\gamma)\big).
\end{equation}
Therefore,
\[
\mahler(f_{\gamma})
=\sum_{\chi\in\Phi} \log^{+}\absvalue{\chi(\gamma)}
\leq
\sum_{\chi\in\Phi} \mahler\big(\chi(\gamma)\big)
=\mahler(\gamma),
\]
proving the first inequality. For the second inequality, for each \(\chi(\gamma)\),
\[
\Weilheight\big(\chi(\gamma)\big)\leq\mahler\big(\chi(\gamma)\big)\leq [\fieldk_{\gamma}:\Q]\cdot \Weilheight\big(\chi(\gamma)\big),
\]
so summing over \(\chi\in\Phi\),
\[
\Weilheight(\gamma)\leq\mahler(\gamma)\leq[\fieldk_{\gamma}:\Q]\cdot \Weilheight(\gamma).\qedhere
\]
\end{proof}

The previous lemma shows that if \(\gamma\) has small Mahler measure or small height, then the characteristic polynomial of \(\Ad\gamma\) also has small Mahler measure.

\begin{remark}
There are several definitions of height and Mahler measure for \(\gamma\) that may be possible (see \cite{Breuillard2011} for instance). Instead of combining the Mahler measures of the \(\chi(\gamma)\)'s, one could also simply define \(\mahler(\gamma)=[\fieldk:\Q]\cdot\Weilheight(\gamma)\) as \(\gamma\in\bfG(\fieldk)\). However, the \(\chi(\gamma)\)'s are not in general in \(\fieldk\), but in \(\fieldk_{\gamma}\).
\end{remark}

\subsection{The Lehmer problem in higher dimensions}

To prove our main theorem, we will need some additional tools from Diophantine geometry. Amoroso and David \cite{AmorosoDavid1999} proved a higher-dimensional version of Drobrowolski' theorem \cite{Dobrowolski1979}. The following corollary of their result will be sufficient for our purposes. Let
\[
\kappa(n)=(n+1)\big[(n+1)!\big]^{n}-n.
\]

\begin{theorem}[Amoroso-David {\cite[Theorem 1.6]{AmorosoDavid1999}}]
\label{theorem:Amoroso-David}
For any \(n\geq 1\), there exists a real number \(c(n)>0\), such that for any element \(\bfalpha\) of \(\Gm^{n}(\SepClosure{\Q})\), whose coordinates are multiplicatively independent, we have
\[
\prod_{i=1}^{n}\Weilheight(\alpha_{i})\geq\frac{c(n)}{D}\log(3D)^{-n\kappa(n)},
\qtwhere
D=[\Q(\alpha_{1},\hdots, \alpha_{n}):\Q].
\]
\end{theorem}

\section{Multiplicative dependence of roots}
\label{section:3}

In this section and the next, we prove \cref{theorem:A,theorem:B}. We keep the notations of \cref{section:2}.
The following proposition constitutes the main ingredient to prove \crefrange{theorem:A}{theorem:D}:

\begin{proposition}
\label{proposition:3.1}
For every \(\epsilon>0\), there exist \(N\in\N\) such that the following holds:
\begin{starenv}
Let \(\gamma\in\Gamma\) be of infinite order. If \(\sqIndex{\fieldk}{\Q}\geq N\) and \(\mahler(\gamma)\leq \epsilon\), then for any distinct \(\chi,\chi'\in\Phi\), if neither \(\chi(\gamma)\) nor \(\chi'(\gamma)\) is a root of unity, then \(\chi(\gamma)\) and \(\chi'(\gamma)\) are \emph{multiplicatively dependent}%
\footnote{The roots of unity are emphasized separately because we treat them differently in the proof.}.
\end{starenv}
\end{proposition}

This will follow from the more quantitative \cref{proposition:3.4}. To prove it, we will use \cref{theorem:Amoroso-David}.

Let \(\Q\subset\fieldL\) be a Galois extension. Let \(\alpha\in\fieldL\). Define
\begin{equation}\label{eq:defw}
w_{\fieldL/\Q}(\alpha)\defeq \frac{\Set{\sigma\in\Hom_{\Q}(\fieldL,\C)\st \absvalue{\sigma(\alpha)}\neq 1}}{[\fieldL:\Q]}.
\end{equation}
This is the proportion of embeddings \(\fieldL\into\C\) mapping \(\alpha\) outside the unit circle.

Let \([\fieldL:\Q]=n\). Let \(\sigma_{1},\hdots,\sigma_{n}\) be the \(n\) distinct embeddings of \(\fieldL\) into \(\C\). The multiplicity of each conjugate of \(\alpha\) in the list
\[
\sigma_{1}(\alpha), \hdots,\sigma_{n}(\alpha)
\]
is \([\fieldL:\Q(\alpha)]\), so
\[
[\fieldL:\Q]\cdot w_{\fieldL/\Q}(\alpha)\geq
\sqIndex{\fieldL}{\Q(\alpha)}=\frac{[\fieldL:\Q]}{[\Q(\alpha):\Q]}.
\]
Thus,
\begin{equation}\label{eq:wlb}
[\Q(\alpha):\Q]\geq \frac{1}{w_{\fieldL/\Q}(\alpha)}.
\end{equation}

\begin{lemma}
\label{lemma:DegLB}
Let \(\gamma\in\Gamma\) be an infinite order semisimple element. Let \(\bfT\) be a maximal \(\fieldk\)-torus containing \(\gamma\) and let \(\chi\in\Phi(\bfG,\bfT)\). Then there is a constant \(C_{0}\) depending only on \(\dim G\) such that \([\Q(\chi(\gamma)):\Q]\geq C_{0}\cdot [\fieldk:\Q]\).
\end{lemma}

\begin{proof}
Since \(\alpha=\chi(\gamma)\) is an eigenvalue of \(\Ad\gamma\), it is an algebraic integer. We can view \(\Ad\gamma\) via the restriction of scalars \(\WeilRes_{\fieldk/\Q}\bfG\) using the diagonal embedding which induces the isomorphism \(\WeilRes_{\fieldk/\Q}\bfG(\Q)\simeq \bfG(\fieldk)\). In this case, in a suitable basis of the Lie algebra, \(\Ad\gamma\), viewed in \(\WeilRes_{\fieldk/\Q}(\bfG)\) is
\[
\big(\presup{\sigma}{(\Ad\gamma)}\tst \sigma\in\Hom_{\Q}(\fieldk,\C)\big).
\]
Recall that \(\Gamma\) is a lattice when diagonally embedded in \(\bfG(\fieldk\otimes\R)\simeq \WeilRes_{\fieldk/\Q}(\bfG)(\R)\). In this case, if we denote by \(V_{c}\) (resp. \(V_{nc}\)) the set of embeddings \(\sigma:\fieldk\into\C\) such that \(\bfG(\fieldk_{\sigma})\) is compact (resp. non-compact), it follows that
\[
\absvalue{\sigma(\alpha)}=1\qtforall \sigma\in V_{c}.
\]
We now derive an upper bound for \(w_{\fieldk_{\gamma}/\Q}(\alpha)\) (defined in \eqref{eq:defw}). Each embedding \(\fieldk\into\C\) has exactly \([\fieldk_{\gamma}:\fieldk]\) extensions to \(\fieldk_{\gamma}\), and if \(\sigma\in V_{c}\), then \(\sigma(\alpha)\) is an eigenvalue for the component of \(\Ad\gamma\) in a compact group, so \(\absvalue{\sigma(\alpha)}=1\). Thus,
\[
\card\Set{\sigma\in\Hom_{\Q}(\fieldk_{\gamma},\C)\st \absvalue{\sigma(\alpha)}=1}\geq \absvalue{V_{c}}\cdot [\fieldk_{\gamma}:\fieldk].
\]
Therefore,
\begin{align*}
\card&\Set{\sigma\in\Hom_{\Q}(\fieldk_{\gamma},\C)\st \absvalue{\sigma(\alpha)}\neq 1}\\
&=[\fieldk_{\gamma}:\Q]-\card\Set{\sigma\in\Hom_{\Q}(\fieldk_{\gamma},\C)\st \absvalue{\sigma(\alpha)}= 1}\\
&\leq [\fieldk_{\gamma}:\Q]-\absvalue{V_{c}}\cdot [\fieldk_{\gamma}:\fieldk]\\
& =[\fieldk_{\gamma}:\fieldk]\big([\fieldk:\Q]-\absvalue{V_{c}}\big)\\
& =[\fieldk_{\gamma}:\fieldk]\cdot\absvalue{V_{nc}}.
\end{align*}
Thus,
\begin{align*}
w_{\fieldk_{\gamma}/\Q}(\alpha)
&=\frac{\card\Set{\sigma\in\Hom_{\Q}(\fieldk_{\gamma},\C)\st \absvalue{\sigma(\alpha)}\neq 1}}{[\fieldk_{\gamma}:\Q]}\\
&\leq \frac{[\fieldk_{\gamma}:\fieldk]\cdot\absvalue{V_{nc}}}{[\fieldk_{\gamma}:\Q]}\\
&= \frac{\absvalue{V_{nc}}}{[\fieldk:\Q]}.
\end{align*}
By \eqref{eq:wlb} it follows that
\[
[\Q(\alpha):\Q]\geq \big(w_{\fieldk_{\gamma}/\Q}(\alpha)\big)^{-1}\geq \frac{[\fieldk:\Q]}{\absvalue{V_{nc}}}.
\]
By construction of \(\Gamma\), \(|V_{nc}|\) is precisely the number of simple factors of \(G\). As such, it is bounded in terms of  \(\dim G\) alone.
\end{proof}

If \(\alpha=\chi(\gamma)\) and \(\beta=\chi'(\gamma)\), since \([\Q(\alpha,\beta):\Q]\geq [\Q(\alpha):\Q]\), we get:

\begin{corollary}
\label{lemma:lower-bound-degree-number-field}
If \(\alpha=\chi(\gamma)\) and \(\beta=\chi'(\gamma)\), then
\[
[\Q(\alpha,\beta):\Q]\geq C_{0}\cdot [\fieldk:\Q].
\]
\end{corollary}

With appropriate bounds on \([\Q(\alpha):\Q]\) and \([\Q(\alpha,\beta):\Q]\) in terms of \([\fieldk:\Q]\), we can now prove the following strengthening of \cref{proposition:3.1}.

\begin{proposition}
\label{proposition:3.4}
Let \(\chi,\chi'\in\Phi\) be distinct characters and let \(\alpha=\chi(\gamma)\), \(\beta=\chi'(\gamma)\). If \(\alpha\) and \(\beta\) are multiplicatively independent, then there exist constants \(C_{1}\) and \(C_{2}\) depending on \(\dim G\) such that
\[
\mahler(\alpha)\cdot\mahler(\beta)\geq C_{1}\cdot \sqIndex{\fieldk}{\Q}\cdot\big(\log \sqIndex{\fieldk}{\Q}\big)^{-C_{2}}.
\]
\end{proposition}


\begin{proof}[Proof of \cref{proposition:3.4}]
Let \(\alpha=\chi(\gamma)\), \(\beta=\chi'(\gamma)\). If they are multiplicatively independent, by \cref{theorem:Amoroso-David} (for \(n=2\)), there are absolute constants \(c_{1},c_{2}>0\) such that
\begin{equation}
\Weilheight(\alpha)\cdot \Weilheight(\beta) \geq c_{1}\frac{\big(\log \sqIndex{\Q(\alpha,\beta)}{\Q}\big)^{-c_{2}}}{\sqIndex{\Q(\alpha,\beta)}{\Q}}.
\end{equation}
Since \(\alpha\) and \(\beta\) belong to the ring of integers of \(\fieldk_{\gamma}\),
\[
\Q(\alpha,\beta)\subset\fieldk(\alpha,\beta)\subset \fieldk_{\gamma}\subset\fieldk_{\bfT},
\]
so
\[
[\Q(\alpha,\beta):\Q]\leq[\fieldk_{\bfT}:\Q]
= [\fieldk_{\bfT}:\fieldk]\cdot[\fieldk:\Q].
\]
On the other hand, applying \cref{lemma:lower-bound-degree-number-field}, there is a constant \(C_{0}\) depending only on \(\dim G\) such that
\[
[\Q(\alpha,\beta):\Q]\geq C_{0}\cdot [\fieldk:\Q].
\]
Let us write \(d_{\bfT}=[\fieldk_{\bfT}:\fieldk]\). We have
\begin{equation}
C_{0}\cdot [\fieldk:\Q]
\leq[\fieldK:\Q]
\leq d_{\bfT}\cdot[\fieldk:\Q],\qtwhere
\fieldK \tis \Q(\alpha),~\Q(\beta),\tor \Q(\alpha,\beta).
\end{equation}
Thus, the upper bound on \([\Q(\alpha,\beta):\Q]\) yields
\[
\Weilheight(\alpha)\cdot \Weilheight(\beta) \geq c_{1}\cdot \frac{(\log d_{\bfT}\cdot\sqIndex{\fieldk}{\Q})^{-c_{2}}}{d_{\bfT}\cdot\sqIndex{\fieldk}{\Q}}.
\]
Since \(d_{\bfT}\) depends only on \(\dim\bfG\), there are constants \(C_{1}\) and \(C_{2}\), depending only on \(\dim G\), such that
\begin{equation}
\Weilheight(\alpha)\cdot \Weilheight(\beta) \geq C_{1}\cdot \frac{(\log \sqIndex{\fieldk}{\Q})^{-C_{2}}}{\sqIndex{\fieldk}{\Q}}.
\end{equation}
But \cref{lemma:lower-bound-degree-number-field} gives lower bounds on \([\Q(\alpha):\Q]\) and \([\Q(\beta):\Q]\) in terms of \([\fieldk:\Q]\), so
\[
\Weilheight(\alpha)\cdot \Weilheight(\beta)
=
\frac{\mahler(\alpha)\cdot \mahler(\beta)}{[\Q(\alpha):\Q]\cdot[\Q(\beta):\Q]}
\leq
\frac{\mahler(\alpha)\cdot\mahler(\beta)}{C_{0}^{2}\cdot[\fieldk:\Q]^{2}}.
\]
We obtain
\[
C_{0}^{2}\cdot C_{1}\cdot \frac{(\log \sqIndex{\fieldk}{\Q})^{-C_{2}}}{\sqIndex{\fieldk}{\Q}}
\leq \frac{\mahler(\alpha)\cdot \mahler(\beta)}{[\fieldk:\Q]^{2}},
\]
and thus,
\[
\mahler(\alpha)\cdot\mahler(\beta)\geq C_{0}^{2}\cdot C_{1}\cdot \sqIndex{\fieldk}{\Q}\cdot\big(\log \sqIndex{\fieldk}{\Q}\big)^{-C_{2}}.\qedhere
\]
\end{proof}

Since \(\mahler(\gamma)\geq \mahler(\alpha)\) and \(\mahler(\gamma)\geq \mahler(\beta)\), we obtain:

\begin{corollary}
\label{corollary:3.5}
Let \(\gamma\in \bfG(\cO_{\fieldk})\). There exist constants \(C_{1},C_{2}>0\) depending only on \(\dim G\) such that for any two distinct roots \(\chi,\chi'\in \Phi\), if \(\chi(\gamma)\) and \(\chi'(\gamma)\) are multiplicatively independent, then
\[
\big(\mahler(\gamma)\big)^{2}\geq C_{1}\cdot \sqIndex{\fieldk}{\Q}\cdot\big(\log \sqIndex{\fieldk}{\Q}\big)^{-C_{2}}.
\]
\end{corollary}

We conclude that there is a constant \(0<\delta<1/2\) depending only on \(\dim G\) such that
\begin{equation}
\mahler(\gamma)\geq \big(\sqIndex{\fieldk}{\Q}\big)^{\frac{1}{2}-\delta},
\qtfor [\fieldk:\Q]~\text{sufficiently large}.
\end{equation}
Thus, if \(\mahler(\gamma)\) is small, by choosing \(\fieldk\) of sufficiently high degree over \(\Q\), we see that any two (distinct) roots must be multiplicatively dependent.




\section{Constructing Salem numbers from short geodesics}

In this section, we prove \cref{theorem:B} of which \cref{theorem:A} is a special case.  Given a short geodesic on \(\Gamma\backslash X\) corresponding to an infinite order semisimple element \(\gamma\), \cref{proposition:3.1} shows that we have two distinct roots \(\chi,\chi'\in\Phi(\bfG,\bfT)\) such that \(\chi(\gamma)\) and \(\chi'(\gamma)\) are multiplicatively dependent. We use this multiplicative dependence together with the action of the Galois group on the root system to construct a Salem number.
Throughout this section, we assume that the assumptions of \cref{proposition:3.1} hold. 
In addition, we assume that \(\Gamma\) is cocompact, which is equivalent to \(\bfG\) being \(\fieldk\)-anisotropic. In this case, every element in \(\Gamma\) is semisimple \cite[11.14]{Raghunathan1972}.

\subsection{Partition of the set of eigenvalues}

Let \(\gamma\in \bfG(\fieldk)\). Define
\[
S_{\gamma}\defeq \Set{\chi(\gamma)\st \chi\in\Phi(\bfG,\bfT)}
= S_{\gamma}^{\hyp}\sqcup S_{\gamma}^{\elliptic}\subset \fieldk_{\gamma},
\]
where \(S_{\gamma}^{\hyp}\) (resp. \(S_{\gamma}^{\elliptic}\)) denotes the values afforded by the roots of \(\gamma\) of absolute value \(\neq 1\) (resp. \(=1\)):
\[
S_{\gamma}^{\elliptic}\defeq \Set{z\in S_{\gamma} \st z ~\text{is torsion}},\quad
S_{\gamma}^{\hyp}\defeq S_{\gamma}\setminus S_{\gamma}^{\elliptic}.
\]
Let \(\scrG_{\fieldk}=\Gal(\fieldk)\) be the absolute Galois group of \(\fieldk\). We have an action of \(\scrG_{\fieldk}\) on \(S_{\gamma}\), \(S_{\gamma}^{\hyp}\) and \(S_{\gamma}^{\elliptic}\).

We will need a simple lemma:

\begin{lemma}
\label{lemma:prod1}
Suppose \(\Gamma\) is cocompact and let \(\alpha\in S_{\gamma}\). Then, \(\prod_{\sigma\in\Gal(\fieldk_{\bfT}/\fieldk)}\sigma(\alpha)=1\).
\end{lemma}

\begin{proof}
Let \(\alpha=\chi(\gamma)\) for some \(\chi\in \Phi\). The assumption that \(\Gamma\) is cocompact means that \(\bfG\) is \(\fieldk\)-anisotropic, i.e. has no \(\fieldk\)-split tori. In particular, \(\Characters(\bfT)^{\Gal(\fieldk_{\bfT}/\fieldk)}=\{1\}.\) Since \(\prod_{\sigma}\chi^\sigma\in \Characters(\bfT)^{\Gal(\fieldk_{\bfT}/\fieldk)},\) we get \(\prod_{\sigma}\chi^\sigma(\gamma)=\prod_{\sigma}\sigma(\alpha)=1\), where all the products are over \(\sigma\in \Gal(\fieldk_{\bfT}/\fieldk)\).
\end{proof}
%
%
%
%
%
%
%
%

\subsection{A Cocycle for the absolute Galois group}

Let \(\alpha=\chi(\gamma)\in S_\gamma^\hyp\). Since \(\sigma(\alpha)\in S_{\gamma}^{\hyp}\), the numbers \(\alpha\) and \(\sigma(\alpha)\) are multiplicatively dependent by \cref{proposition:3.1}. We infer that there exist integers \(p\), \(q\), such that \(\sigma(\alpha)^{p}=\alpha^{q}\). Define the cocycle
\begin{align*}
c:\scrG_{\fieldk}\times S_{\gamma}^{\hyp}&\to \Q^{\times}\\
c(\sigma,\alpha)&=\frac{p}{q},\qtif \sigma(\alpha)^{p}=\alpha^{q}.
\end{align*}
Let us show that \(c\) is well-defined.
\begin{proof}
Let \(p,p',q,q'\in\Z\) be such that
\[
\sigma(\alpha)^{p}=\alpha^{q}\qand
\sigma(\alpha)^{p'}=\alpha^{q'}.
\]
Then,
\begin{alignat*}{2}
\alpha^{q\cdot p'}
& = (\alpha^{q})^{p'} = \big(\sigma(\alpha)^{p}\big)^{p'}
 = \sigma(\alpha)^{p\cdot p'} = \big(\sigma(\alpha)^{p'}\big)^{p} = \alpha^{q'\cdot p}.
\end{alignat*}
Hence, \(\alpha^{qp'-q'p}=1\). Since \(\alpha\in S_{\gamma}^{\hyp}\), it is not a root of unity, so we have \(qp'=q'p\) and thus \(p'/q'=p/q\).
\end{proof}

\begin{lemma}
The map \(c\) is a 1-cocycle, i.e., if \(\sigma,\tau\in \scrG_{\fieldk}\) and \(\alpha\in S_{\gamma}^{\hyp}\), the relation
\begin{equation}
c(\tau\sigma,\alpha)=c\big(\tau,\sigma(\alpha)\big)\cdot c(\sigma,\alpha)
\end{equation}
holds.
\end{lemma}

\begin{proof}
Let \(p,q,a,b\in\Z^{*}\) be such that
\[
c(\sigma,\alpha)=p/q
\qand
c\big(\tau,\sigma(\alpha)\big)=a/b.
\]
We have
\begin{alignat*}{3}
\sigma(\alpha)^{p}
& = \alpha^{q}&&\qLeftrightarrow&
c(\sigma,\alpha)
& = p/q,\\
\tau\big(\sigma(\alpha)\big)^{a}
& = \sigma(\alpha)^{b}&&\qLeftrightarrow&
c\big(\tau,\sigma(\alpha)\big)
& = a/b.
\end{alignat*}
Thus,
\begin{align*}
(\tau\sigma)(\alpha)^{ap}
= \big[\big(\tau(\sigma(\alpha))\big)^{a}\big]^{p}
= \big[\sigma(\alpha)^{b}\big]^{p}
= \big[\sigma(\alpha)^{p}\big]^{b}
= \big[\alpha^{q}\big]^{b}.
\end{align*}
So
\[
(\tau\sigma)(\alpha)^{ap}=\alpha^{qb},
\]
that is,
\[
c(\tau\sigma,\alpha)=\frac{ap}{qb}=\frac{a}{b}\cdot\frac{p}{q}=c\big(\tau,\sigma(\alpha)\big)\cdot c(\sigma,\alpha),
\]
which is the cocycle relation.
\end{proof}

\begin{lemma}
The map \(c\) takes values in \(\{-1,+1\}\).
\end{lemma}

\begin{proof}
Let \(p,q\in\Z^{*}\) be such that \(c(\sigma,\alpha)=p/q\). Then, \(\sigma(\alpha)^{p}=\alpha^{q}\). Since \(\sigma\) is a homomorphism,
\begin{align*}
\big[\sigma^{2}(\alpha)\big]^{p^{2}}
& = \big[\sigma(\sigma(\alpha))\big]^{p^{2}}
= \big[\sigma\big(\sigma(\alpha)\big)^{p}\big]^{p}
= \big[\sigma\big(\sigma(\alpha)^{p}\big)\big]^{p}\\
& = \big[\sigma\big(\alpha^{q}\big)\big]^{p}
= \big[\sigma(\alpha)\big]^{pq}
= \big[\sigma(\alpha)^{p}\big]^{q}
= \alpha^{q^{2}}.
\end{align*}
Hence,
\[
\big[\sigma^{2}(\alpha)\big]^{p^{2}}=\alpha^{q^{2}},
\qqtext{or, equivalently,}
c(\sigma^{2},\alpha)=c(\sigma,\alpha)^{2}.
\]
The cocycle relation gives
\[
c(\sigma^{2},\alpha)=c\big(\sigma,\sigma(\alpha)\big)\cdot c(\sigma,\alpha)=c(\sigma,\alpha)^{2},
\]
so we get
\begin{equation}
\label{eq:sigma-cocycle}
c(\sigma,\alpha)=c\big(\sigma,\sigma(\alpha)\big).
\end{equation}
Consider the identity
\begin{equation}
\label{eq:cocycle-induction-hyp}
c(\sigma^{n},\alpha)=c(\sigma,\alpha)^{n}\quad (\alpha\in S_{\gamma}^{\hyp},~n\in\N).
\end{equation}
We prove it by induction on \(n\). We have already shown that it holds for \(n=1,2\) and all \(\alpha\in S_{\gamma}^{\hyp}\). Assume that \eqref{eq:cocycle-induction-hyp} holds for some \(n\geq 2\) and all \(\alpha\in S_{\gamma}^{\hyp}\). By the cocycle relation,
\begin{alignat*}{2}
c(\sigma^{n+1},\alpha)
& = c\big(\sigma^{n},\sigma(\alpha)\big)\cdot c(\sigma,\alpha)\\
& = c(\sigma,\sigma(\alpha))^{n}\cdot c(\sigma,\alpha)
&&\quad(\text{by the induction hypothesis \eqref{eq:cocycle-induction-hyp}})\\
&= c(\sigma,\alpha)^{n}\cdot c(\sigma,\alpha)
&&\quad(\text{by \eqref{eq:sigma-cocycle}})\\
&= c(\sigma,\alpha)^{n+1}.
\end{alignat*}
Thus, we have
\[
c(\sigma^{n},\alpha)=c(\sigma,\alpha)^{n}\qtforall n\in\N.
\]
This means that
\[
\frac{p_{n}}{q_{n}}= \left(\frac{p}{q}\right)^{n}\qtwhere
\sigma^{n}(\alpha)^{p_{n}}=\alpha^{q_{n}}\tand p_{1}/q_{1}=p/q.
\]
It follows that
\[
\big(\sigma^{n}(\alpha)\big)^{p^{n}}=\alpha^{q^{n}}\qtforall n\in\N^{*}.
\]
Now, for \(\sigma\in\scrG_{\fieldk}\), the orbit \(\set{\sigma^{n}\tst n\in\N}\) is finite, so the corresponding sequence \(\big((p/q)^{n}\big)_{n}\) must be finite as well. This implies that \(c(\sigma,\alpha)\in\{1,-1\}\).
\end{proof}

This shows that if \(\sigma(\alpha)\) and \(\alpha\) are multiplicatively dependent, there is an integer \(p\) such that \(\sigma(\alpha)^{p}=\alpha^{\pm p}\).

\begin{proposition}
\label{4.5}
For each \(\sigma\in \scrG_{\fieldk}\), let \(p_{\sigma}\) be defined by
\begin{equation}
p_{\sigma}\defeq \min \set[big]{n\in\N^{*} \tst \sigma(\alpha)^{n}=\alpha^{n} \tor \sigma(\alpha)^{n}=\alpha^{-n}}
\end{equation}
and define \(\ulalpha\) by
\begin{equation}
\label{eq:ulalpha}
\ulalpha\defeq
\alpha^{\prod_{\sigma\in \scrG_{\fieldk}} p_{\sigma}}.
\end{equation}
Then, for any \(\sigma\in\scrG_{\fieldk}\), we have
\[
\sigma(\ulalpha)=\ulalpha^{c(\sigma,\alpha)}.
\]
\end{proposition}

\begin{proof}
For ease of notation, let \(P \defeq \prod_{\sigma\in\scrG_{\fieldk}}p_{\sigma}\). Then, \(\ulalpha=\alpha^{P }\), and for any \(\sigma\in\scrG_{\fieldk}\),
\[
P /p_{\sigma}=\prod_{\substack{\tau\in\scrG_{\fieldk}\\ \tau\neq \sigma}} p_{\tau}.
\]
Let \(\sigma\in \scrG_{\fieldk}\). Then,
\begin{align*}
\sigma(\ulalpha)
& = \sigma\left(\alpha^{P }\right)
= \sigma\left(\alpha^{p_{\sigma}\cdot(P /p_{\sigma})}\right)
= \sigma\left(\big(\alpha^{p_{\sigma}}\big)^{P /p_{\sigma}}\right)\\
& = \sigma\big(\alpha^{p_{\sigma}}\big)^{P /p_{\sigma}}.
\end{align*}
We know that \(c(\sigma,\alpha)=\pm 1\). For any integer \(p\), \(\sigma(\alpha)^{p}=\sigma(\alpha^{p})\) and \(\alpha^{-p}=(\alpha^{p})^{-1}\), so
\[
\sigma(\alpha^{p})=\begin{cases}
\alpha^{p}&\tif c(\sigma,\alpha)=+1,\\
(\alpha^{p})^{-1}&\tif c(\sigma,\alpha)=-1.
\end{cases}
\]
Thus, we always have
\begin{equation}
\sigma(\alpha^{p})=\big(\alpha^{p}\big)^{c(\sigma,\alpha)}
\qtwhenever
\sigma(\alpha)^{p}=\alpha^{\pm p}.
\end{equation}
In particular, this holds for \(p_{\sigma}\), and we get
\[
\sigma(\alpha^{p_{\sigma}})=(\alpha^{p_{\sigma}})^{c(\sigma,\alpha)}\qtforall \sigma\in\scrG_{\fieldk}.
\]
We can then rewrite \(\sigma(\ulalpha)\) as
\begin{align*}
\sigma(\ulalpha)
& = \sigma\big(\alpha^{p_{\sigma}}\big)^{P /p_{\sigma}}
= \big[(\alpha^{p_{\sigma}})^{c(\sigma,\alpha)}\big]^{P /p_{\sigma}}
= \big[(\alpha^{c(\sigma,\alpha)})^{p_{\sigma}}\big]^{P /p_{\sigma}}\\
& = \big[\alpha^{c(\sigma,\alpha)}\big]^{P }
= \big[\alpha^{P }\big]^{c(\sigma,\alpha)}
= \ulalpha^{c(\sigma,\alpha)}.
\end{align*}
This proves the claim.
\end{proof}

One would like to take \(\ulalpha\) as the candidate for the Salem number appearing implicitly in \cref{theorem:A} (resp. \cref{theorem:B}) but the high power \(P \defeq \prod_{\sigma\in \scrG_{\fieldk}} p_{\sigma}\) in the definition of \(\ulalpha\) is difficult to control as the degree of the trace field grows. To circumvent this issue we replace \(\ulalpha\) by a smaller product \(\Tilde{\alpha}\), which will exhibit similar properties under the Galois action. Let \(\fieldk'\defeq \fieldk(\ulalpha)\). By \cref{4.5}, this is a quadratic extension of \(\fieldk\). Put
\begin{equation}
\label{eq:tilde-alpha}
\tilde{\alpha}=\prod_{\sigma\in \Gal(\fieldk_{\bfT}/\fieldk')}\sigma(\alpha)\in \fieldk_{\gamma}.
\end{equation}
We have
\[
\tilde{\alpha}^{P }=\prod_{\sigma\in \Gal(\fieldk_{\bfT}/\fieldk')}\sigma(\ulalpha)=\ulalpha^{[\fieldk_{\bfT}:\fieldk']}=\alpha^{P [\fieldk_{\bfT}:\fieldk']}\neq 1.
\]
By definition of \(\tilde{\alpha}\) and \cref{lemma:prod1}, it is clear that
\begin{align*}
\sigma(\tilde{\alpha})&=\begin{cases}
\tilde\alpha & \tif \sigma\in \Gal(\fieldk_{\bfT}/\fieldk')\\
\tilde{\alpha}^{-1} & \tif \sigma\in \Gal(\fieldk_{\bfT}/\fieldk)\setminus \Gal(\fieldk_{\bfT}/\fieldk')
\end{cases}\\
&=\tilde{\alpha}^{c(\sigma,\alpha)}.
\end{align*}
We note that, by \cref{lemma:prod1} and the fact that \(\alpha\) is non-torsion, \(\tilde\alpha\) cannot be fixed by \(\Gal(\fieldk_{\bfT}/\fieldk)\), so \(\fieldk'=\fieldk(\tilde\alpha)\neq \fieldk\).

\begin{corollary}
\label{4.7}
If \(\alpha\in S_{\gamma}^{\hyp}\), then \(\tilde{\alpha}\) is an algebraic unit of signature \((s,t)\) where \(0\leq t\leq r_{2}\) and \(1\leq s+2t\leq r_{1}+2r_{2}\). In particular, if \(\bfG\) is \(\R\)-simple, then \(\tilde{\alpha}\) is a Salem number.
\end{corollary}

\begin{proof}
Consider \(\eta\defeq \tilde{\alpha}+\tilde{\alpha}^{-1}\), where \(\tilde{\alpha}\) is defined by \eqref{eq:tilde-alpha}. The group \(\scrG_{\fieldk}\) fixes \(\eta\) so it is an element of \(\fieldk\). For each real embedding \(\tau:\fieldk\into \R\) such that \(\presup{\tau}{\bfG}(\R)\) is compact, we have \(\tau(\eta)\in[-2,2]\). Since \(\tilde{\alpha}\) is not torsion we have at least one embedding \(\tau\colon \fieldk\into \C\) (it may be that \(\tau:\fieldk\into\R\)) such that \(\tau(\eta)\notin[-2,2]\). We conclude that \(\tilde{\alpha}\) is an algebraic unit of signature \((s,t)\), where
\[
0\leq t\leq r_{2}
\qand
1\leq s+2t\leq r_{1}+2r_{2}.
\]
In particular, in case \(\bfG\) is \(\R\)-simple, \((r_{1},r_{2})=(1,0)\), so \(\tilde{\alpha}\) is an algebraic unit of signature \((1,0)\), in other words, a Salem number.
\end{proof}

\begin{proof}[Proof of \cref{theorem:A} (resp. \cref{theorem:B})]
Assume that the \hyperref[Salem-conjecture]{Salem conjecture} (resp. the \hyperref[Lehmer-conjecture-signature]{Lehmer conjecture of signature \((r_{1},r_{2})\)}) holds with lower bound \(\epsilon>0\), and let \(\fieldk\), \(\bfG\), \(\Gamma\) be as in \cref{section:3}, with the additional asusmption that \(\Gamma\) is cocompact.

Since \hyperref[Margulis-conjecture']{Margulis' conjecture} is known whenever the degree of the trace field is fixed, we can always assume that \([\fieldk:\Q]\) is large enough.
By \cref{proposition:3.1}, we may choose \([\fieldk:\Q]\) large enough so that for any \(\gamma\in\Gamma\) of infinite order, any two infinite order distinct eigenvalues of \(\Ad\gamma\) are multiplicatively dependent. Denote by \(\alpha\) one such infinite order eigenvalue. By \cref{4.7}, \(\tilde{\alpha}=\prod_{\sigma} \sigma(\alpha)\) where the product is over \(\sigma\in\Gal(\fieldk_{\bfT}/\fieldk(\ulalpha))\), is a Salem number (resp. an algebraic unit of signature \((s,t)\) where \(0\leq t\leq r_{2}\) and \(1\leq s+2t\leq r_{1}+2r_{2}\)). In addition, \(\tilde{\alpha}^{P}=\alpha^{P[\fieldk_{\bfT}:\fieldk]/2}\neq 1\), since \(\alpha\in S_{\gamma}^{\hyp}\), and
\[
\epsilon<
\mahler(\tilde{\alpha})\leq \tfrac{1}{2}[\fieldk_{\bfT}:\fieldk]\cdot \mahler(\gamma)\leq \tfrac{1}{2}(\dim \bfG- \rk \bfG)!\cdot \mahler(\gamma)
\]
Thus, the set
\[
U=\Set{\gamma\in\Gamma \st \mahler(\gamma)<2\epsilon/(\dim \bfG- \rk \bfG)!}
\]
is an open neighborhood of the identity in \(\Gamma\) such that \(\Gamma\cap U\) consists of elements of finite order.
\end{proof}

\begin{remark}
The proof actually shows, under the assumptions of \cref{theorem:A}, that the eigenvalues of elements \(\Ad\gamma\) with short translation lengths are \textit{virtually} Salem numbers, in the sense that some large enough power is either a Salem number or 1. Indeed, \(\ulalpha\) as defined in \eqref{eq:ulalpha} can also be shown to be a Salem number -- the same proof as that of \cref{4.7} works verbatim.
\end{remark}

\section{Structure of the Bottom of the Length Spectrum}
\label{section:5}

In this last section, we prove a more precise form of \cref{theorem:C} and \cref{theorem:D}. Recall the formula \eqref{eq:length} for the length \(\ell(\gamma)\) of a closed geodesic corresponding to a semisimple element \(\gamma\).

\begin{theorem}
Let \(G\) be a semisimple Lie group without compact factors. There exist constants \(C_{1},C_{2}\), dependent only on \(G\), with the following property. Let \(\Gamma\subset G\) be an arithmetic lattice with trace field \(\fieldk\) and let \(\gamma_{1},\gamma_{2}\) be two semisimple elements of \(\Gamma\) such that \(\ell(\gamma_{1})^{2}\) and \(\ell(\gamma_{2})^{2}\) are linearly independent over \(\Q\). Then
\[
\ell(\gamma_{1})+\ell(\gamma_{2})\geq 2C_{1}^{1/2}\cdot \sqIndex{\fieldk}{\Q}^{1/2}\big(\log \sqIndex{\fieldk}{\Q}\big)^{-C_{2}/2}.
\]
In particular, the translation lengths of \(\gamma_{1}\) and \(\gamma_{2}\) cannot be simultaneously small.
\end{theorem}

\begin{proof}
We adopt the notation from \cref{section:2}. Let \(\bfT_{1},\bfT_{2}\) be maximal tori of \(\bfG\), defined over \(\fieldk\), containing \(\lambda_{1},\lambda_{2}\) respectively.
We first claim that there are roots \(\lambda_{1}\in\Phi(\bfG,\bfT_{1})\) and \(\lambda_{2}\in\Phi(\bfG,\bfT_{2})\), such that \(\lambda_{1}(\gamma_{1})\) and \(\lambda_{2}(\gamma_{2})\) are multiplicatively independent. Indeed, if not, there is a real number \(\alpha_{\sigma}\) depending on \(\sigma\), and rational numbers \(q_{i,\lambda}\) such that \(\log |\lambda(\gamma_{i})|_{\sigma}= q_{i,\lambda}\cdot\alpha_{\sigma}\) for any embedding \(\sigma:\fieldk\into\C\) and any \(\lambda\in\Phi(\bfG,\bfT_{i})\) (\(i=1,2\)). Then, we have
\[
\ell(\gamma_{i})^{2}=\sum_{\sigma\in V_{\fieldk}} \sum_{\lambda\in\Phi(\bfG,\bfT_{i})} q_{i,\lambda}^{2}\cdot\alpha_{\sigma}^{2}=\left(\sum_{\lambda\in \Phi(\bfG,\bfT_{i})}q_{i,\lambda}^{2}\right)\left(\sum_{\sigma\in V_{\fieldk}}\alpha_{\sigma}^{2}\right).
\]
It is then clear that \(\ell(\gamma_{1})^{2}\) and \(\ell(\gamma_{2})^{2}\) are linearly dependent over \(\Q\), a contradiction, so the claim follows.

Both \(\lambda_{1}(\gamma_{1})\) and \(\lambda_{2}(\gamma_{2})\) are algebraic integers of degree at most \(\dim G\) over \(\fieldk\), so an application of \cref{theorem:Amoroso-David} and \cref{lemma:DegLB} yields
\[
\mahler\big(\lambda_{1}(\gamma_{1})\big)\cdot\mahler\big(\lambda_{2}(\gamma_{2})\big)\geq C_{1}\cdot \sqIndex{\fieldk}{\Q}\cdot\big(\log \sqIndex{\fieldk}{\Q}\big)^{-C_{2}}.
\]
The theorem follows by the AM--GM inequality together with the inequalities \(\mahler(\gamma_{i})\geq \mahler(\lambda_{i}(\gamma_{i}))\) and the Cauchy-Schwarz inequality \(\ell(\gamma_{i})\gg \mahler(\gamma_{i})\).
\end{proof}

With a similar method, we can prove that the semisimple elements generating higher rank tori admit a uniform lower bound on their translation length (\cref{theorem:D}).

\begin{theorem}
\label{theorem:5.2}
Assume that \(G\) is a noncompact simple (real or complex) Lie group. Let \(\gamma\) be a semisimple element of an arithmetic lattice \(\Gamma\subset G\) with trace field \(\fieldk\). let \(S\) be the connected component of the Zariski closure of the subgroup generated by \(\gamma\). If \(\rk S\geq 2\) then
\[
\ell(\gamma)\geq C_{1}^{1/2}\cdot \sqIndex{\fieldk}{\Q}^{1/2}\big(\log \sqIndex{\fieldk}{\Q}\big)^{-C_{2}/2}.
\]
\end{theorem}

\begin{proof}
As in the previous proof, we adopt the notation from \cref{section:2}. Let \(\bfT\) be a maximal torus of \(\bfG\) defined over \(\fieldk\), such that \(\gamma\in \bfT(\fieldk)\) and let \(\bfS\) be the connected component of the Zariski closure of the group generated by \(\gamma\) in \(\bfG\). We fix the embedding \(\fieldk\into \fieldK\), \(\fieldK=\R\) or \(\C\) for which \(\bfG(\fieldK)\) is equal to \(G\). In this way, the rank of \(\bfS\) as a \(\fieldk\)-torus in \(\bfG\) coincides with the absolute rank of \(S\) as a \(\fieldK\)-torus in \(G\). We can consider two cases, either all the numbers \(\lambda(\gamma), \lambda\in \Phi(\bfG,\bfT)\) are multiplicatively dependent, or there is a pair \(\lambda_{1}(\gamma),\lambda_{2}(\gamma)\) which are multiplicatively independent. In the latter case we argue using \cref{theorem:Amoroso-David} and \cref{lemma:DegLB}, just like before, that
\[
\mahler\big(\lambda_{1}(\gamma)\big)\cdot \mahler\big(\lambda_{2}(\gamma)\big)\geq C_{1}\cdot \sqIndex{\fieldk}{\Q}\cdot\big(\log \sqIndex{\fieldk}{\Q}\big)^{-C_{2}}.
\]
In this case, an application of AM-GM inequality and \(\mahler(\gamma)\geq \max\{\mahler(\lambda_{1}(\gamma)),\mahler(\lambda_{2}(\gamma))\}\) and the Cauchy-Schwarz inequality \(\ell(\gamma)\gg \mahler(\gamma)\) finishes the proof.

It remains to show that in the first case, the rank of \(\bfS\) is \(1\). The torus \(\bfT\) contains \(\bfS\) so every character of \(\bfS\) is the restriction of a character of \(\bfT\). Let \(\xi_{1},\ldots,\xi_{d}\in\Characters(\bfT)\) be characters restricting to a basis of \(\Characters(\bfS)\). The roots in \(\Phi(\bfG,\bfT)\) generate a finite index subgroup of \(\Characters(\bfT)\) so there exists \(m_{1},m_{2}\in\Z\setminus\{0\}\) such that
\[
\xi_{i}^{m_{i}}\in \< \Phi(\bfG,\bfT)\> \quad (i=1,2).
\]
Suppose \(d\geq 2\). Since all eigenvalues of \(\Ad\gamma\) are multiplicatively dependent, there exist non-zero integers \(s_{1}, s_{2}\), such that \(\xi_{1}(\gamma)^{s_{1}}\xi_{2}(\gamma)^{s_{2}}=1\). The group generated by \(\gamma\) is Zariski dense in \(\bfS\) so this implies that \(\xi_{1}^{s_{1}}\xi_{2}^{s_{2}}=1\) on \(\bfS\), which is impossible as \(\xi_{1},\xi_{2}\) are basis elements. This implies that \(d=1\), which means that the rank of \(\bfS\) is \(1\).
\end{proof}

\begin{remark}
A closer look at the proof of \cref{theorem:5.2} shows that the statement may be extended to semisimple Lie groups \(G=G_{1}\times\cdots\times G_{r}\), where the \(G_i\)'s are the simple (real or complex) factors. In this case, the rank condition on the torus \(S=S_{1}\times \cdots\times S_{r}\) is to be replaced by \(\rk S_{i}\geq 2\) for some \(i=1,\ldots,r\).
\end{remark}


\end{document}